\documentclass{imanumARXIV}
\usepackage{graphicx}
\usepackage{subfig}
\usepackage{verbatim}

\newcommand{\norm}[1]{\left\| #1 \right\|}
\newcommand{\khlsum}[1]{\sum_{K_{h}^{l} \in \mathcal{T}_{h}^{l}}\int_{K_{h}^{l}} #1\ \mbox{d}\sigma }
\newcommand{\khsum}[1]{\sum_{K_{h} \in \mathcal{T}_{h}}\int_{K_{h}} #1\ \mbox{d}\sigma_{h} }
\newcommand{\ehlsum}[1]{\sum_{e_{h}^{l} \in \mathcal{E}_{h}^{l}}\int_{e_{h}^{l}} #1\ \mbox{d}s }

\begin{document}

\title{Adaptive discontinuous Galerkin methods on surfaces}
\shorttitle{Adaptive DG methods on surfaces}

\author{%
{\sc
Andreas Dedner {\sc and}
Pravin Madhavan\thanks{Corresponding author. Email: p.madhavan@warwick.ac.uk}} \\
Mathematics Institute and Centre for Scientific Computing, University of Warwick,\\
Coventry CV4 7AL, UK
}
\shortauthorlist{Dedner \& Madhavan}

\maketitle

\begin{abstract}
{
We present a dual weighted residual-based a posteriori error estimate for a
discontinuous Galerkin (DG) approximation of a linear second-order elliptic
problem on compact smooth connected and oriented surfaces in
$\mathbb{R}^{3}$ which are implicitly represented as level sets of a smooth
function. We show that the error in the energy norm may be split into a ``residual part'' and a higher order ``geometric part''. Upper and lower bounds
for the resulting a posteriori error estimator are proven
and we consider a number of challenging test problems to demonstrate the
reliability and efficiency of the estimator. We also present a novel
``geometric'' driven refinement strategy for PDEs on surfaces which
considerably improves the performance of the method on complex surfaces.
}
{discontinuous galerkin; interior penalty; a posteriori error estimation; adaptive refinement; surface PDEs.}
\end{abstract}

\section{Introduction}
\label{sec:Introduction}
Partial differential equations (PDEs) on manifolds have become an active area of research in recent years due to the fact that, in many applications, models have to be formulated not on a flat Euclidean domain but on a curved surface. For example, they arise naturally in fluid dynamics (e.g.~surface active agents on the interface between two fluids, \cite{JamLow04}) and material science (e.g.~diffusion of species along grain boundaries, \cite{DecEllSty01}) but have also emerged in areas as diverse as image processing and cell biology (e.g.~cell motility involving processes on the cell membrane, \cite{neilson2011modelling} or phase separation on biomembranes, \cite{EllSti10}). 

Finite element methods (FEM) for elliptic problems and their error analysis have been successfully applied to problems on surfaces via the intrinsic approach in \cite{dziuk1988finite} based on interpolating the surface by a triangulated one. This approach has subsequently been extended to parabolic problems in \cite{dziuk2007surface} as well as evolving surfaces in \cite{dziuk2007finite}. The literature on the application of FEM to various surface PDEs and geometric flows is now quite extensive, reviews of which can be found in \cite{dziuk2013finite} and \cite{deckelnick2005computation}. However, there are a number of situations where FEM may not be the appropriate numerical method, for instance, advection dominated problems which lead to steep gradients or even discontinuities in the solution. 

DG methods are a class of numerical methods that have been successfully applied to hyperbolic, elliptic and parabolic PDEs arising from a wide range of applications. Some of its main advantages compared to `standard' finite element methods include the ability of capturing discontinuities as arising in advection dominated problems, and less restriction on grid structure as well as on the choice of basis functions, which make them ideal for a posteriori error estimation and hp-adaptive refinement. The main idea of DG methods is not to require continuity of the solution between elements. Instead, inter-element behaviour has to be prescribed carefully in such a way that the resulting scheme has adequate consistency, stability and accuracy properties. A short introduction to DG methods for both ODEs and PDEs is given in \cite{cockburn2003discontinuous}. A history of the development of DG methods can be found in \cite{cockburn2000development} and \cite{arnold2002unified}. \cite{arnold2002unified} provides an in-depth analysis of a large class of discontinuous Galerkin methods for second-order elliptic problems. 

DG methods have first been extended to surfaces in \cite{dedner2012analysis}, where an interior penalty (IP) method for a linear second-order elliptic problem was introduced and optimal a priori error estimates in the $L^{2}$ and $DG$ norms for piecewise linear ansatz functions and surface approximations were derived. \cite{larsson2013continuous} have considered a continuous/discontinuous Galerkin method for a fourth order elliptic PDE on surfaces. \cite{ju2009finite}, \cite{lenz2011convergent} and \cite{GieMue_prep} have also derived a priori error bounds for finite volume methods on (evolving) surfaces via the intrinsic approach.

The literature on a posteriori error estimation and adaptivity on surfaces is significantly less extensive than its a priori counterpart. \cite{demlow2008adaptive} derived an a posteriori error estimator for the finite element discretisation of the Laplace-Beltrami operator on surfaces, showing that the error can be split into a residual indicator term and a geometric error term. In a similar fashion, \cite{ju2009posteriori} derived an estimator for the finite volume discretisation of a steady convection-diffusion-rection equation on surfaces. \cite{mekchay2011afem} have considered an adaptive finite element method for the Laplace-Beltrami operator on $C^{1}$ graphs. In light of the benefits of DG methods for hp-adaptivity, it is natural to extend the DG framework to surfaces and this is a first step towards that direction.

This paper is organised in the following way. We consider a linear second-order elliptic equation on a compact smooth connected and oriented surface $\Gamma \subset \mathbb{R}^{3}$ and consider an interior penalty (IP) method on a triangulated surface $\Gamma_{h}$, introduced in \cite{dedner2012analysis}. We derive a dual weighted residual-based a posteriori error estimator, and show its reliability and efficiency in the energy norm. The estimator has a similar structure to the one derived for surface FEM in \cite{demlow2008adaptive}, with both a standard \emph{residual} term and a higher order \emph{geometric} residual.

We then present some numerical results, making use of the Distributed and
Unified Numerics Environment (DUNE) software package (see
\cite{dunegridpaperII:08}, \cite{dunegridpaperI:08}) and, in particular,
the DUNE-FEM module described in \cite{dunefempaper:10} (also see
{\em dune.project.org/fem} for more details on this module). We consider a number of test problems which numerically verify the reliability and efficiency of the estimator. In the process, we also present a computationally efficient adaptive refinement strategy which makes use of the geometric residual of the estimator. 

\section{Notation and Setting}\label{sec:NotationAndSetting}
The notation in this section closely follows that used in \cite{demlow2008adaptive} and \cite{dedner2012analysis}.
\subsection{Smooth surface $\Gamma$ and problem formulation.}
Let $\Gamma$ be a connected $C^{2}$ compact smooth and oriented surface in $\mathbb{R}^{3}$ given by the zero level set of a signed distance function $|d(x)| = dist(x,\Gamma)$ defined on an open subset $U$ of $\mathbb{R}^{3}$. For simplicity we assume that $\partial \Gamma = \emptyset$ and that $d < 0$ on the interior of $\Gamma$ and $d > 0$ on the exterior. The outward unit normal $\nu$ of $\Gamma$ is thus given by
\[\nu(\xi) = \nabla d(\xi),\ \xi \in \Gamma. \]
With a slight abuse of notation we also denote the projection to $\Gamma$ by $\xi$, i.e.~$\xi:U \rightarrow \Gamma$ is given by
\begin{equation}\label{eq:uniquePoint}
\xi(x) = x - d(x)\nu(x) \quad \mbox{where } \nu(x):=\nu(\xi(x)).
\end{equation}
It is worth noting that such a projection is (locally) unique provided that the width $\delta_{U} > 0$ of $U$ satisfies
\[\delta_{U} < \left[\max_{i=1,2}\norm{\kappa_{i}}_{L^{\infty}(\Gamma)}\right]^{-1}\]
where $\kappa_{i}$ denotes the $i$th principle curvature of the Weingarten map given by $\mathbf{H}(x) := \nabla^{2}d(x)$. Throughout this paper, we denote by 
\[ \mathbf{P} := \mathbf{I} - \nu \otimes \nu, \]
the projection onto the tangent space $T_{\xi}\Gamma$ on $\Gamma$ at a point $\xi \in \Gamma$. Here $\otimes$ denotes the usual tensor product.
\begin{definition}
For any function $\eta$ defined on an open subset of $U$ containing $\Gamma$ we can define its \emph{tangential gradient} on $\Gamma$ by
\[ \nabla_{\Gamma}\eta := \nabla \eta - (\nabla \eta \cdot \nu) \nu = \mathbf{P}\nabla \eta \]
and then the \emph{Laplace-Beltrami} operator on $\Gamma$ by 
\[ \Delta_{\Gamma} \eta := \nabla_{\Gamma}\cdot (\nabla_{\Gamma} \eta).
\]
\end{definition}
\begin{definition}
We define the surface Sobolev spaces
\[ H^{m}(\Gamma) := \{u \in L^{2}(\Gamma) \ : \ D^{\alpha}u \in L^{2}(\Gamma)\ \forall |\alpha| \leq m \}, \quad m \in \mathbb{N} \cup \{ 0 \}, \]
with corresponding Sobolev seminorm and norm respectively given by
\[ 
|u|_{H^{m}(\Gamma)} := \left(\sum_{|\alpha|=m} \norm{D^{\alpha}u}_{L^{2}(\Gamma)}^{2}\right)^{1/2}, \quad 
\norm{u}_{H^{m}(\Gamma)} := \left(\sum_{k=0}^m |u|_{H^{k}(\Gamma)}^{2}\right)^{1/2}.
\]
\end{definition}     
We refer to \cite{wlokapartial} for a proper discussion of Sobolev spaces on manifolds. 

The problem that we consider in this paper is the following second-order elliptic equation:
\begin{equation}\label{eq:EllipticGamma}
-\Delta_{\Gamma} u + u = f
\end{equation}
for a given $f \in L^{2}(\Gamma)$. Using integration by parts on surfaces the weak problem reads:
\\
\\
$(\mathbf{P}_{\Gamma})$ Find $u \in H^{1}(\Gamma)$ such that  
\begin{equation} \label{eq:weakH1}
 a_{\Gamma}(u,v) =
   \int_\Gamma fv\ \mbox{d}\sigma \quad\forall v\in H^1(\Gamma) 
\end{equation}
where
\[ a_{\Gamma}(u,v) := \int_\Gamma \nabla_{\Gamma}u\cdot \nabla_{\Gamma}v + u v\ \mbox{d}\sigma.  \]
Existence and uniqueness of a solution $u$ follows from standard arguments. See \cite{aubin1982nonlinear} and \cite{wlokapartial} for further details.

\subsection{Discrete problem on $\Gamma_{h}$}
The smooth surface $\Gamma$ is approximated by a polyhedral surface $\Gamma_{h} \subset U$ composed of planar triangles. Let $\mathcal{T}_{h}$ be the associated regular conforming triangulation of $\Gamma_{h}$ i.e.
\[ \Gamma_{h} = \bigcup_{K_{h} \in \mathcal{T}_{h}} K_{h}. \]
Let $\nu_{h}$ denote the outward unit normal on $\Gamma_{h}$, and let $\mathcal{N}$ denote the set of nodes of triangles in $\mathcal{T}_{h}$. The vertices are taken to sit on $\Gamma$, i.e. $\mathcal{N} \subset \Gamma$. We assume that $\xi : \Gamma_{h} \rightarrow \Gamma$ is bijective and that $\nu \cdot \nu_{h} \geq 0$ everywhere on $\Gamma_{h}$. We also denote by $h_{K_{h}}$ the largest edge of $K_{h} \in \mathcal{T}_{h}$. Given $p \in \mathcal{N}$, we define the patch $w_{p} = interior(\cup_{K_{h} | p \in \bar{K}_{h}}\bar{K}_{h})$ and let $h_{p} = \max_{K_{h} \subset w_{p}}h_{K_{h}}$. Let $\mathcal{E}_{h}$ denote the set of all codimension one intersections of elements $K_h^{+},K_h^{-}\in \mathcal{T}_{h}$ (i.e., the edges) and denote by $h_{e_{h}}$ the length scale associated with an edge $e_{h} \in \mathcal{E}_{h}$. We define the conormal $n_h^{+}$ on such an intersection $e_{h} \in \mathcal{E}_{h}$ of elements $K_h^{+}$ and $K_h^{-}$ by demanding that \\[1mm]
$\bullet$ $n_h^{+}$ is a unit vector,\\
$\bullet$ $n_h^{+}$ is tangential to (the planar triangle) $K_h^{+}$, \\
$\bullet$ in each point $x \in e_h$ we have that $n_h^{+} \cdot (y-x) \leq 0$ for all $y \in K_h^{+}$. \\[1mm]
Analogously one can define the conormal $n_h^{-}$ on $e_h$ by exchanging $K_h^{+}$ with $K_h^{-}$. Note that, in general, $n_{h}^{+} \not = -n_{h}^{-}$.

Let $\hat{K} \subset \mathbb{R}^{2}$ be the reference element and let $F_{K_{h}} : \hat{K}\rightarrow K_{h} \subset  \mathbb{R}^{3}$ for $K_{h} \in \mathcal{T}_{h}$. We define the DG space associated with $\Gamma_h$ by 
\begin{align*}
V_{h} =\{ v_{h} \in L^2(\Gamma_h): v_{h}|_{K_{h}}= \hat{v} \circ F_{K_{h}}^{-1} \text{ for some } \hat{v} \in \mathbb{P}^{k}(\hat{K})\ \ \ \forall K_h \in \mathcal{T}_h\}.
\end{align*}
In addition, we define the vector-valued function space
\[ \Sigma_{h} := \{ \tau_{h} \in [L^2(\Gamma_h)]^{3}: \tau_{h}|_{K_{h}}= \nabla F_{K_{h}}^{-T}\left( \hat{\tau} \circ F_{K_{h}}^{-1} \right) \text{ for some } \hat{\tau} \in [\mathbb{P}^{k}(\hat{K})]^{2}\ \ \ \forall K_h\in \mathcal{T}_h \}. \]
Here, $\nabla F_{K_{h}}^{-1}$ refers to the (left) \emph{pseudo-inverse} of $\nabla F_{K_{h}}$ i.e.  
\[\nabla F_{K_{h}}^{-1} = \left(\nabla F_{K_{h}}^{T}\nabla F_{K_{h}}\right)^{-1}\nabla F_{K_{h}}^{T}.\]
For $v_h\in V_{h}$, let 
\[
v_h^{+/-}:= v_h \big{|}_{\partial K_h^{+/-}}.
\]
\begin{definition}\label{def:jumpAverageOparators}
Let $\tilde{n} \in \mathbb{R}^{3}$. For $q \in \Pi_{K_{h} \in \mathcal{T}_{h}}L^{2}(\partial K_{h}) $, let $\{q\}$ and $[q]$ by given by
\[ \{q\} := \frac{1}{2}(q^{+}+q^{-}),\   \ [q] := q^{+} - q^{-}\   \ \mbox{on}\ e_{h} \in \mathcal{E}_{h}. \]  
For $\varphi \in [\Pi_{K_{h} \in \mathcal{T}_{h}}L^{2}(\partial K_{h})]^{3}$, $\{\varphi;\tilde{n}\}$ and $[\varphi;\tilde{n}]$ are given by 
\[ \{\varphi;\tilde{n}\} := \frac{1}{2}(\varphi^{+} \cdot \tilde{n}^{+} - \varphi^{-} \cdot \tilde{n}^{-}),\   \ [\varphi;\tilde{n}] := \varphi^{+}\cdot \tilde{n}^{+}+\varphi^{-}\cdot \tilde{n}^{-}\   \ \mbox{on}\ e_{h} \in \mathcal{E}_{h}. \]
\end{definition}
We now formulate our discrete problem on $\Gamma_{h}$ for a given function $f_{h} \in L^{2}(\Gamma_{h})$ (note that, in general, this is not a finite element function, it will be related to the function $f$ given in problem $(\mathbf{P}_{\Gamma})$ later on, see (\ref{eq:rel_f_fh}) below):
\\
\\
$\left(\mathbf{P}_{\Gamma_{h}}^{IP}\right)$ Find $u_{h} \in V_{h}$  such that
\begin{equation} \label{eq:InteriorPenaltyGammah}
a_{\Gamma_{h}}^{IP}(u_{h},v_{h}) = \sum_{K_{h} \in \mathcal{T}_{h}}\int_{K_{h}}f_{h} v_{h}\ \mbox{d}\sigma_{h}\ \forall v_{h} \in V_{h}
\end{equation}
where
\begin{align}\label{eq:InteriorPenaltyGammahForm}
a_{\Gamma_{h}}^{IP}(u_{h},v_{h}) &:= \sum_{K_{h} \in \mathcal{T}_{h}}\int_{K_{h}}\nabla_{\Gamma_{h}}u_{h}\cdot \nabla_{\Gamma_{h}}v_{h} + u_{h} v_{h}\ \mbox{d}\sigma_{h} - \sum_{e_{h} \in \mathcal{E}_{h}}\int_{e_{h}}[u_{h}]\{\nabla_{\Gamma_{h}}v_{h}; n_{h} \} + [v_{h}]\{\nabla_{\Gamma_{h}}u_{h}; n_{h}\}\ \mbox{d}s_{h}\notag \\ 
&+ \sum_{e_{h} \in \mathcal{E}_{h}}\int_{e_{h}}\beta_{e_{h}}[u_{h}][v_{h}]\ \mbox{d}s_{h}
\end{align}
is the (symmetric) IP method considered in \cite{dedner2012analysis}. Note
that our definition of the jump and average operators depend on the
a-priori choice of $K_h^+,K_h^-$ but the bilinear form is independent of
this choice since only products of these operators occur. The penalty parameters
$\beta_{e_{h}}$ are given by $\beta_{e_{h}} = \omega_{e_{h}}h_{e_{h}}^{-1}$ where $h_{e_{h}}$ is some length scale
associated with the intersection $e_{h}$ (for instance, the edge length). The interior penalty parameters $\omega_{e_{h}}$ are uniformly bounded with respect to $h := \max_{e_{h} \in \mathcal{E}_{h}} h_{e_{h}}$. 

We now define a norm on the space of piecewise smooth functions: 
\begin{definition}\label{def:DGNorm}
For $u_{h} \in V_{h}$ we define
\[ |u_{h}|_{1,h}^{2} := \sum_{K_{h} \in \mathcal{T}_{h}} \norm{u_{h}}_{H^{1}(K_{h})}^{2}\    \ , \    \ |u|_{*,h}^{2} := \sum_{e_{h} \in \mathcal{E}_{h}} h_{e_{h}}^{-1}  \norm{[u_{h}]}_{L^{2}(e_{h})}^{2}.\]
The $DG$ norm on $\Gamma_{h}$ is given by
\[ \norm{u_{h}}_{DG(\Gamma_{h})}^{2} := |u_{h}|_{1,h}^{2} + |u_{h}|_{*,h}^{2}. \]
\end{definition}
Boundedness and stability of (\ref{eq:InteriorPenaltyGammahForm}) follow in a similar fashion as for the classical, planar IP method (see \cite{arnold2002unified} for more details) provided that the penalty parameters $\omega_{e_{h}}$ are large enough. This is because $\Gamma_{h}$ is composed of planar triangles.

\subsection{Surface lifting onto $\Gamma$}
\begin{definition}
For any function $w$ defined on $\Gamma_{h}$ we define the \emph{surface lift} onto $\Gamma$ by
\[ w^{l}(\xi) := w(x(\xi)),\ \xi \in \Gamma, \]
where $x(\xi)$ is defined as the unique solution of
\[ x = \xi + d(x)\nu(\xi).\]
\end{definition}
In particular, we 
\begin{equation} \label{eq:rel_f_fh}
\mbox{define }f_{h} \mbox{ such that } f_{h}^{l} = f \mbox{ on } \Gamma.
\end{equation}
We also denote by $\tilde{w}^{-l}$ the \emph{inverse} surface lift onto $\Gamma_{h}$ of some function $\tilde{w}$ defined on $\Gamma$, satisfying $(\tilde{w}^{-l})^l = \tilde{w}$. Furthermore, for every $K_{h} \in \mathcal{T}_{h}$, there is a unique curved triangle $K_{h}^{l} := \xi(K_{h}) \subset \Gamma$. We now define the regular, conforming triangulation $\mathcal{T}_{h}^{l}$ of $\Gamma$ such that
\[ \Gamma = \bigcup_{K_{h}^{l} \in \mathcal{T}_{h}^{l}} K_{h}^{l}. \] 
The triangulation $\mathcal{T}_{h}^{l}$ of $\Gamma$ is thus induced by the triangulation $\mathcal{T}_{h}$ of $\Gamma_{h}$ via the surface lift. Similarly, $e_{h}^{l}:=\xi(e_{h}) \in \mathcal{E}_{h}^{l}$ are the unique curved edges.

The appropriate function space for surface lifted functions are given by
\[V_{h}^{l} := \{v_{h}^{l} \in L^{2}(\Gamma)\ : \ v_{h}^{l}(\xi) = v_{h}(x(\xi))\ \mbox{for some}\ v_{h} \in V_{h} \}\]
and
\[\Sigma_{h}^{l} := \{ \tau_{h}^{l} \in [L^2(\Gamma)]^{3}: \tau_{h}^{l}(\xi) = \nabla \xi^{-T} \tau_{h}(x(\xi)) \text{ for some } \tau_{h} \in \Sigma_{h}\}.\]
Note that the DG norm for functions $u_{h}^{l} \in V_{h}^{l}$ is the same one as in Definition \ref{def:DGNorm} but with the triangulation $\mathcal{T}_{h}^{l}$ instead and corresponding length scale $h_{e_{h}^{l}}$ associated with $e_{h}^{l}$. We will denote this norm by $DG(\Gamma)$.

We define for $x \in \Gamma_{h}$
\[\mathbf{P}_{h}(x) = \mathbf{I} - \nu_{h}(x) \otimes \nu_{h}(x)\]
so that, for $v_{h}$ defined on $\Gamma_{h}$ and $x \in \Gamma_{h}$,
\[\nabla_{\Gamma_{h}}v_{h}(x) = \mathbf{P}_{h}\nabla v_{h}(x).\]
Finally, one can show that for $x \in \Gamma_{h}$ and $v_{h}$ defined on $\Gamma_{h}$, we have
\begin{equation}\label{eq:Gammah2GammaGradient} 
\nabla_{\Gamma_{h}} v_{h}(x) = \mathbf{P}_{h}(x)(\mathbf{I} - d\mathbf{H})(x)\mathbf{P}(x)\nabla_{\Gamma} v_{h}^{l}(\xi(x))  
\end{equation}
and
\begin{equation}\label{eq:Gamma2GammahGradient}
\nabla_{\Gamma} v_{h}^{l}(\xi(x)) = \mathbf{F}_{h}(x) \nabla_{\Gamma_{h}}v_{h}(x)  
\end{equation}
where $\mathbf{F}_{h}(x) := (\mathbf{I}-d\mathbf{H})(x)^{-1} \left(\mathbf{I} - \frac{\nu_{h} \otimes \nu}{\nu_{h} \cdot \nu} \right)$. Next we state integral equalities which we shall use repeatedly. For $x \in \Gamma_{h}$, let
\begin{equation}\label{eq:changeOfMeasure}
\delta_{h}(x) \mbox{d}\sigma_{h}(x) = \mbox{d}\sigma(\xi(x)), \  \
\delta_{e_{h}}(x) \mbox{d}s_{h}(x) = \mbox{d}s(\xi(x)).
\end{equation}
Note that, by construction, $\delta_{h}(x), \delta_{e_{h}}(x)  > 1$ for all $x$. Also, let
\begin{equation}\label{eq:Ah}
\mathbf{A}_{h}(x) = \mathbf{A}_{h}^{l}(\xi(x))  = \frac{1}{\delta_{h}(x)}\mathbf{P}(x) (\mathbf{I} - d\mathbf{H})(x) \mathbf{P}_{h}(x) (\mathbf{I} - d\mathbf{H})(x) \mathbf{P}(x). 
\end{equation}
Then one can show that
\[ \int_{\Gamma_{h}} \nabla_{\Gamma_{h}}u_{h} \cdot \nabla_{\Gamma_{h}}v_{h}\ \mbox{d}\sigma_{h} = \int_{\Gamma} \mathbf{A}_{h}^{l} \nabla_{\Gamma}u_{h}^{l} \cdot \nabla_{\Gamma}v_{h}^{l}\ \mbox{d}\sigma.  \]
It is worth noting that the geometric quantities $d$ and $\mathbf{H}$ in (\ref{eq:Ah}) are always evaluated on the discrete surface $\Gamma_{h}$.
Finally, we define 
\begin{align}\label{eq:Bh}
 \mathbf{B}_{h} := \sqrt{\delta_{h}} (\mathbf{P} - \mathbf{A}_{h})\mathbf{F}_{h} 
\end{align}
which will be used in the error estimates later on.

We can derive explicit formulas for the quantities $\delta_{h}$ and $\delta_{e_{h}}$ 
defined above.
\begin{lemma}
Assume that $x \in \Gamma_{h}$. Then
\begin{align}
&\delta_{h}(x) = (1-d(x)\kappa_{1}(x))(1 - d(x) \kappa_{2}(x))\nu \cdot \nu_{h}, \\
&\delta_{e_{h}}(x) = \norm{\nabla \xi(x) \tau_{h}(x)}_{l^{2}}
\end{align}
where $\tau_{h}$ is the unit tangent on $e_{h}$.
\end{lemma}
\begin{proof}
See \cite{demlow2008adaptive} for the first expression. To prove the second expression, we do the following: let $\hat{e} \subset \mathbb{R}$ be a reference edge for codimension one entities. Let $f : \hat{e} \rightarrow e_{h} \subset \Gamma_{h}$ be the linear transformation from the reference edge to the edge $e_{h} \in \mathcal{E}_{h}$. $F := f' \in \mathbb{R}^{3 \times 1}$ is tangent to $e_{h}$ and so $F = \lambda \tau_{h}$ where $\lambda \in \mathbb{R}$. Hence we have that $\mbox{d}s = \sqrt{|F^{T}\nabla \xi^{T}\nabla \xi F|} \mbox{d} \hat{s} = |\lambda||\nabla \xi \tau_{h}| \mbox{d} \hat{s}$ where $ \nabla \xi \in \mathbb{R}^{3 \times 3}$ is the gradient of the projection mapping $\xi$ given in (\ref{eq:uniquePoint}) and $\mbox{d} \hat{s}$ is the Lebesgue measure on $\hat{e}$. Similarly, we have that  $\mbox{d}s_{h} = |\lambda| \mbox{d} \hat{s}$ and the second expression follows.  
\end{proof}

\subsection{Cl\'{e}ment interpolant}
We now define a quasi-interpolant and state some estimates that it must satisfy. Given $z \in L^{1}(\Gamma)$ and $p \in \mathcal{N}$, we let
\begin{equation}\label{eq:ClementInterpolantWeights}
z^{-l}_{p} := \frac{1}{\int_{w_{p}}\varphi_{p}\ \mbox{d}\sigma_{h}}\int_{w_{p}} \varphi_{p} z^{-l}\ \mbox{d}\sigma_{h}  
\end{equation}
where $\varphi_{p} \in V_{h} \cap H^{1}(\Gamma_{h})$ denotes the Lagrange nodal basis function associated with $p$, and define
\begin{equation}\label{eq:ClementInterpolant} 
I_{h} z^{-l} = \sum_{p \in \mathcal{N}}z_{p}^{-l} \varphi_{p}. 
\end{equation}
We note a useful property that the weights $z^{-l}_{p}$ satisfy (see (2.2.33) in \cite{demlow2008adaptive}):
\begin{equation}\label{eq:ClementInterpolantWeightsProperty}
\norm{z^{-l}_{p}}_{L^{2}(w_{p})} \leq \sqrt{\frac{3}{2}} \norm{z^{-l}}_{L^{2}(w_{p})} \leq \sqrt{\frac{3}{2}}\norm{\sqrt{\delta_{h}}}_{L^{\infty}(w_{p})} \norm{z}_{L^{2}(w_{p}^l)}.  
\end{equation}
Since $\{ \varphi_{p} \}$ is a partition of unity i.e. $\sum_{p \in \mathcal{N}}\varphi_{p} =1$, we also have the following:
\begin{equation}\label{eq:ClementInterpolantProperty}
\int_{\Gamma_{h}}(z^{-l} - I_{h} z^{-l})\ \mbox{d}\sigma_{h} = \sum_{p \in \mathcal{N}}\int_{w_{p}}(z^{-l} - z^{-l}_{p})\varphi_{p}\ \mbox{d}\sigma_{h} = 0. 
\end{equation}
The Clement interpolant satisfies the following estimates.
\begin{lemma}\label{thm:PoicareInequalityShapeRegular}
Let $z \in H^{1}(\Gamma)$. Assume that the mesh $\mathcal{T}_{h}$ is shape-regular and that the number of elements sharing the node $p$ is bounded. Let $w_{p}^{l}$ be the surface lift of the patch $w_{p}$ onto $\Gamma$. Then for each $p \in \mathcal{N}$, we have
\begin{equation}\label{eq:PoicareInequalityL2PatchShapeRegular}
\norm{z^{-l} - z_{p}^{-l}}_{L^{2}(w_{p})} \leq C h_{p}\norm{\mathbf{A}_{h}}^{\frac{1}{2}}_{l^{2},L^{\infty}(w_{p})}\norm{\nabla_{\Gamma}z}_{L^{2}(w_{p}^{l})}.
\end{equation}
Let also $p \in \bar{e}_{h} \subset \mathcal{E}_{h}$. Then
\begin{equation}\label{eq:PoicareInequalityL2edgeShapeRegular}
\norm{z^{-l} - z_{p}^{-l}}_{L^{2}(e_{h})} \leq C h_{p}^{\frac{1}{2}}\norm{\mathbf{A}_{h}}^{\frac{1}{2}}_{l^{2},L^{\infty}(w_{p})}\norm{\nabla_{\Gamma}z}_{L^{2}(w_{p}^{l})}.
\end{equation}
Note that $C$ does not depend on any essential quantities. Here $\norm{\mathbf{A}_{h}}_{l^{2},L^{\infty}(w_{p})} := \norm{\norm{\mathbf{A}_{h}}_{l^{2}\rightarrow l^{2}}}_{L^{\infty}(w_{p})}$. 
\end{lemma}
\begin{proof}
See \cite{demlow2008adaptive}. 
\end{proof}

\section{Dual weighted residual equation}
We derive a residual equation for some quantity of interest $J(u - u_{h}^{l})$ where $J$ is some bounded, linear functional acting on $H^{1}(\Gamma) + V_{h}^{l}$. 

\subsection{Bilinear form on $\Gamma$}
Before we state the bilinear form we consider on $\Gamma$, we require the following DG lifting operators. 
\begin{definition}
Let $w \in H^{1}(\Gamma) + V_{h}^{l}$. Define the operators $\mathbf{L}: H^1(\Gamma)+V_{h}^{l} \rightarrow \Sigma_{h}^{l}$ and, for every $e_{h} \in \mathcal{E}_{h}$, $\mathbf{L}_{e_{h}}: H^1(\Gamma)+V_{h}^{l} \rightarrow \Sigma_{h}^{l}$ by respectively
\begin{align}\label{eq:DGLift}
\khlsum{\mathbf{L}(w) \cdot \phi} = \ehlsum{[w] \{\phi ; n \}},\  \  \   \  \khlsum{\mathbf{L}_{e_{h}^l}(w) \cdot \phi} = \int_{e_{h}^{l}}[w] \{\phi ; n \}\ \mbox{d}s
\end{align}
for all $\phi \in \Sigma_{h}^{l}$, where $n^{+}$ and $n^{-}$ are respectively the unit surface conormals to $K_h^{l+}$ and $K_h^{l-}$ on $e_{h}^{l} \in \mathcal{E}_{h}^{l}$, satisfying $n^{+} = -n^{-}$.
\end{definition}
\begin{remark}\label{LiftingEstimateRemark1}
Note that $w \in H^{1}(\Gamma) \Rightarrow \mathbf{L}(w) = 0$ and $\mathbf{L}_{e_{h}^l}(w) = 0$ for all $e_{h} \in \mathcal{E}_{h}$.  
\end{remark}
\begin{remark}\label{LiftingEstimateRemark2}
Note that, for each $e_{h} \in \mathcal{E}_{h}$, $\mathbf{L}_{e_{h}^l}(w)$ vanishes outside the union of the two triangles containing $e_{h}$ and that $ \mathbf{L}(w)  = \sum_{e_{h} \in \mathcal{E}_{h}} \mathbf{L}_{e_{h}^l}(w)$ for all $w \in H^{1}(\Gamma) + V_{h}^{l}$. 
\end{remark}
The DG lifting operator $\mathbf{L}_{e_{h}^l}$ satisfies the following stability estimate:
\begin{lemma} \label{LiftingEstimate}
Let $w_{e_{h}} = K_{h}^{+} \cup K_{h}^{-}$. Then for each $e_{h}^l = K_{h}^{l+} \cap K_{h}^{l-} \in \mathcal{E}_{h}^l$, there exists a constant $C_{L} > 0$ such that
\[\norm{\mathbf{L}_{e_{h}^l}(w)}_{L^{2}(w_{e_{h}}^{l})}^{2} \leq C_{L} \norm{\mathbf{F}_{h}}_{l^{2},L^{\infty}(e_{h})}^{2}\norm{\mathbf{P}_{h}(\mathbf{I} - d\mathbf{H})\mathbf{P}}_{l^{2},L^{\infty}(w_{e_{h}})}^2 \norm{\delta_{e_{h}}}_{L^{\infty}(e_{h})}^2\norm{\sqrt{\beta_{e_{h}}}[w^{-l}]}_{L^{2}(e_{h})}^{2}\]
for every $w \in V_{h}^{l} + H^{1}(\Gamma)$. The constant $C_{L}$ depends solely on the shape-regularity of the mesh and on the upper bound for the penalty parameters $\omega_{e_h}$.
\end{lemma}
\begin{proof}
The proof will follow a similar argument to the one found in \cite{schotzau2003mixed}. Let $\Sigma_{h}^{l}(w_{e_{h}}^l)$ denote the space of all functions in $\Sigma_{h}^l$ restricted to $w_{e_{h}}^l$. For $w \in V_{h}^{l} + H^{1}(\Gamma)$, making use of Remark \ref{LiftingEstimateRemark1} and (\ref{eq:Gamma2GammahGradient}), we have
\begin{align*}
\norm{\mathbf{L}_{e_{h}^{l}}(w)}_{L^{2}(w_{e_{h}}^{l})} &= \sup_{\phi \in \Sigma_{h}^{l}(w_{e_{h}}^{l})} \frac{\int_{w_{e_{h}}^{l}}\mathbf{L}_{e_{h}^{l}}(w) \cdot \phi\ \mbox{d}\sigma}{\norm{\phi}_{L^{2}(w_{e_{h}}^{l})}} = \sup_{\phi \in \Sigma_{h}^{l}(w_{e_{h}}^{l})} \frac{\int_{e_{h}^{l}}[w] \{\phi ; n\}\ \mbox{d}s}{\norm{\phi}_{L^{2}(w_{e_{h}}^{l})}} \\
&\leq \sup_{\phi \in \Sigma_{h}^{l}(w_{e_{h}}^{l})} \frac{\left(\int_{e_{h}^{l}} \delta_{e_{h}}\beta_{e_{h}}|[w]|^{2}\ \mbox{d}s\right)^{\frac{1}{2}}\left(\int_{e_{h}^{l}} \delta_{e_{h}}^{-1} \beta_{e_{h}}^{-1}|\{\phi ; n\}|^{2}\ \mbox{d}s\right)^{\frac{1}{2}}}{\norm{\phi}_{L^{2}(w_{e_{h}}^{l})}} \\
&\leq \sup_{\phi \in \Sigma_{h}^{l}(w_{e_{h}}^{l})} \frac{\left(\int_{e_{h}}\delta_{e_{h}}^2\beta_{e_{h}}|[w^{-l}]|^{2}\ \mbox{d}s_{h}\right)^{\frac{1}{2}}\left(\int_{e_{h}}{\beta_{e_{h}}^{-1}|\mathbf{F}_{h}\tilde{\phi}|^{2}\ \mbox{d}s_{h}}\right)^{\frac{1}{2}}}{\norm{\phi}_{L^{2}(w_{e_{h}}^{l})}}  
\end{align*}
where $\tilde{\phi} \in \Sigma_{h}$. Applying the trace theorem on $\Gamma_{h}$ and lifting back onto $\Gamma$ using (\ref{eq:Gammah2GammaGradient}), we have for all $\phi \in \Sigma_{h}^{l}(w_{e_{h}}^l)$:
\begin{align*}
  \int_{e_{h}}\beta_{e_{h}}^{-1}|\mathbf{F}_{h}\tilde{\phi}|^{2}\ \mbox{d}s_{h} &\leq  \norm{\mathbf{F}_{h}}_{l^2,L^{\infty}(e_{h})}^2\int_{e_{h}}\beta_{e_{h}}^{-1}|\tilde{\phi}|^{2}\ \mbox{d}s_{h} \leq C \norm{\mathbf{F}_{h}}_{l^2,L^{\infty}(e_{h})}^2 \int_{w_{e_h}}\omega_{e_{h}}^{-1}|\tilde{\phi}|^{2}\ \mbox{d}\sigma_{h} \\
  &\leq C \norm{\mathbf{F}_{h}}_{l^2,L^{\infty}(e_{h})}^2 \norm{\mathbf{P}_{h}(\mathbf{I} - d\mathbf{H})\mathbf{P}}_{l^{2},L^{\infty}(w_{e_{h}})}^2\norm{\phi}_{L^{2}(w_{e_{h}}^{l})}^2 
\end{align*}
where we have used that $\delta_{h}^{-1} < 1$. Here $C$ depends on the shape-regularity of the mesh and on the upper bound for the penalty parameters $\omega_{e_h}$ but not on any other essential quantity like
$h$. This provides the desired estimate.
\end{proof}

We can now define a bilinear form on $\Gamma$ which is well-defined in the space $(H^1(\Gamma)+V_{h}^{l}) \times (H^1(\Gamma)+V_{h}^{l})$ by making use of the DG lifting. Let
\begin{align}\label{eq:InteriorPenaltyGammaForm}
A_{\Gamma}^{IP}(v,z) &:= \khlsum{\nabla_{\Gamma}v \cdot \nabla_{\Gamma}z + vz} - \khlsum{\mathbf{L}(v) \cdot \nabla_{\Gamma} z + \mathbf{L}(z) \cdot \nabla_{\Gamma} v} \notag \\
&+ \ehlsum{\beta_{e_{h}^{l}}[v][z]} 
\end{align}
where $\beta_{e_h^l} := \delta_{e_h}^{-1}\beta_{e_h}$. Provided that the penalty parameters $\omega_{e_{h}}$ are large enough, boundedness and stability of $A_{\Gamma}^{IP}$ in $H^{1}(\Gamma) + V_{h}^{l}$ follow from Lemma 3.4 in \cite{dedner2012analysis}. The bilinear form $A_{\Gamma}^{IP}$ is related to the original problem $\left(\mathbf{P}_{\Gamma}\right)$ in the following way: 

\begin{lemma}\label{thm:dualBilinearForm2Rhs}
Let $u_{h} \in V_{h}$ denote the solution to $(\mathbf{P}_{\Gamma_{h}}^{IP})$ and $u_{h}^{l} \in V_{h}^{l}$ its surface lift onto $\Gamma$. Let $z_{h}^{l} \in V_{h}^{l,c} := V_{h}^{l} \cap H^{1}(\Gamma)$. Then we have
\[ A_{\Gamma}^{IP}(u_{h}^{l},z_{h}^{l}) = \khlsum{f z_{h}^{l}} - E_{h}(z_{h}^{l}) \]
where
\begin{align*}
E_{h}(z_{h}^{l}) &:= \sum_{K_{h}^{l} \in \mathcal{T}_{h}^{l}}\int_{K_{h}^{l}}(\mathbf{A}_{h}^{l}-\mathbf{P})\nabla_{\Gamma}u_{h}^{l} \cdot\nabla_{\Gamma}z_{h}^{l}+\left(\delta_{h}^{-1}-1\right)u_{h}^{l}z_{h}^{l} + \left(1-\delta_{h}^{-1}\right)fz_{h}^{l}\ \mbox{d}\sigma \\
&+ \sum_{e_{h}^{l} \in \mathcal{E}_{h}^{l}}\int_{e_{h}^{l}}[u_{h}^{l}]\left( \{\nabla_{\Gamma}z_{h}^{l}; n\} - \left\{\mathbf{P}_{h}^{l}(\mathbf{I}-d\mathbf{H})\mathbf{P}\nabla_{\Gamma}z_{h}^{l}; n_{h}^{l}\right\} \delta_{e_{h}}^{-1} \right)\ \mbox{d}s 
\end{align*}
\end{lemma}
\begin{proof}
We notice that $A_{\Gamma}^{IP}(u_{h}^{l},z_{h}^{l}) = A_{\Gamma}^{IP}(u_{h}^{l} - u,z_{h}^{l}) + A_{\Gamma}^{IP}(u,z_{h}^{l})$. Since $u, z_{h}^{l} \in H^{1}(\Gamma)$ we have that $A_{\Gamma}^{IP}(u,z_{h}^{l}) = a_{\Gamma}(u,z_{h}^{l}) = \khlsum{f z_{h}^{l}}$ by Remark \ref{LiftingEstimateRemark1} and (\ref{eq:weakH1}). Also, we have that $A_{\Gamma}^{IP}(u - u_{h}^{l},z_{h}^{l}) = E_{h}(z_{h}^{l})$ by Lemma 4.2 in \cite{dedner2012analysis}.
\end{proof} 

\subsection{Residual equation}
In order to derive the residual equation, we consider the following dual problem: find $z \in H^{1}(\Gamma)$ such that   
\begin{align}\label{eq:DualProblem}
A_{\Gamma}^{IP}(v,z) = J(v)\ \forall v \in H^{1}(\Gamma). 
\end{align}
In a similar fashion to \cite{houston2007energy}, we decompose the error $\mbox{e}_{h} := u - u_{h}^{l}$ using $u_{h}^{l} = u_{h}^{l,c} + u_{h}^{l,\perp}$ with $u_{h}^{l,c} \in V_{h}^{l,c}$ and $u_{h}^{l,\perp} \in V_{h}^{l,\perp}$ where $V_{h}^{l,\perp}$ denotes the orthogonal complement in $V_{h}^{l}$ of $V_{h}^{l,c}$ with respect to the DG norm. Thus $\mbox{e}_{h}^{c} = u - u_{h}^{l,c} \in H^{1}(\Gamma)$. Let $z_{h}^{l} \in V_{h}^{l,c}$, from the dual problem (\ref{eq:DualProblem}) we have
\begin{align*}
J(\mbox{e}_{h}) &= J(\mbox{e}_{h}^{c}) - J(u_{h}^{l,\perp}) 
  = A_{\Gamma}^{IP}(\mbox{e}_{h}^{c},z) - J(u_{h}^{l,\perp}) \\ 
 &= A_{\Gamma}^{IP}(\mbox{e}_{h},z) + A_{\Gamma}^{IP}(u_{h}^{l,\perp},z) - J(u_{h}^{l,\perp}) \\
 &= A_{\Gamma}^{IP}(u,z) - A_{\Gamma}^{IP}(u_{h}^{l},z-z_{h}^{l}) 
    - A_{\Gamma}^{IP}(u_{h}^{l},z_{h}^{l}) 
    + A_{\Gamma}^{IP}(u_{h}^{l,\perp},z) - J(u_{h}^{l,\perp}) 
\end{align*}
Using the fact that $A_{\Gamma}^{IP}(u,z) = a(u,z)$ (by Remark \ref{LiftingEstimateRemark1}), (\ref{eq:weakH1}) and Lemma \ref{thm:dualBilinearForm2Rhs}, we get
\begin{align*}
J(e_{h}) &= \khlsum{f(z-z_{h}^{l})} - A_{\Gamma}^{IP}(u_{h}^{l},z-z_{h}^{l}) + A_{\Gamma}^{IP}(u_{h}^{l,\perp},z) - J(u_{h}^{l,\perp}) + E_{h}(z_{h}^{l}).
\end{align*}
Using the fact that $z - z_{h}^{l} \in H^{1}(\Gamma)$ so that $[z-z_{h}^{l}] = 0$ holds, we have
\begin{align*}
J(e_{h}) =  &\khlsum{f(z-z_{h}^{l})} - \khlsum{\nabla_{\Gamma}u_{h}^{l} \cdot \nabla_{\Gamma}(z-z_{h}^{l}) + u_{h}^{l}(z-z_{h}^{l})}\\ 
&+ \khlsum{\mathbf{L}(u_{h}^{l}) \cdot \nabla_{\Gamma} (z-z_{h}^{l})} + A_{\Gamma}^{IP}(u_{h}^{l,\perp},z) - J(u_{h}^{l,\perp}) + E_{h}(z_{h}^{l}).
\end{align*}
Moving the first two integrals in the above onto $\Gamma_{h}$ and integrating by parts, we get
\begin{align*}
J(e_{h}) =  & \sum_{K_{h} \in \mathcal{T}_{h}}\left(\int_{K_h}(f_{h}\delta_{h} + \Delta_{\Gamma_{h}}u_{h} - u_{h}\delta_{h})(z^{-l}-z_{h})\ \mbox{d}\sigma_{h} - \int_{\partial K_{h}}\nabla_{\Gamma_{h}}u_{h} \cdot n_{K_{h}}(z^{-l} - z_{h})\ \mbox{d}s_{h} \right) \notag \\
&- \khlsum{(\mathbf{P}-\mathbf{A}_{h}^{l})\nabla_{\Gamma}u_{h}^{l} \cdot \nabla_{\Gamma}z} + \khlsum{\left(\delta_{h}^{-1}-1\right)(u_{h}^{l} - f)z_{h}^{l}} \notag \\ 
&+ \khlsum{\mathbf{L}(u_{h}^{l}) \cdot \nabla_{\Gamma} (z-z_{h}^{l})} \\
&+ \sum_{e_{h}^{l} \in \mathcal{E}_{h}^{l}}\int_{e_{h}^{l}}[u_{h}^{l}]\left( \{\nabla_{\Gamma}z_{h}^{l}; n\} - \left\{\mathbf{P}_{h}^{l}(\mathbf{I}-d\mathbf{H})\mathbf{P}\nabla_{\Gamma}z_{h}^{l}; n_{h}^{l} \right\}\delta_{e_{h}}^{-1}\right)\ \mbox{d}s + A_{\Gamma}^{IP}(u_{h}^{l,\perp},z) - J(u_{h}^{l,\perp}).
\end{align*}
We now wish to move all the terms in the above onto the discrete surface. Making use of (\ref{eq:Gamma2GammahGradient}), we have the following:
\begin{align*}
- \sum_{K_{h}^{l} \in \mathcal{T}_{h}^{l}}\int_{K_{h}^{l}}(\mathbf{P}-\mathbf{A}_{h}^{l})\nabla_{\Gamma}u_{h}^{l} \cdot \nabla_{\Gamma} z\ \mbox{d}s
&= -\sum_{K_{h} \in \mathcal{T}_{h}} \int_{K_{h}}\delta_{e_{h}}\mathbf{F}_{h}^{T}(\mathbf{P}^{-l} - \mathbf{A}_{h})\mathbf{F}_{h}\nabla_{\Gamma_{h}}u_{h} \cdot \nabla_{\Gamma_{h}} z^{-l}\ \mbox{d}s_{h}
\end{align*}
and
\begin{align*}
&\khlsum{\mathbf{L}(u_{h}^{l}) \cdot \nabla_{\Gamma} (z-z_{h}^{l})} = \khsum{\mathbf{L}^{-l}(u_{h}^{l})\cdot \delta_{h} \mathbf{F}_{h}\nabla_{\Gamma_{h}}(z^{-l} - z_{h})}.
\end{align*} 
Furthermore, making use of the fact that $\mathbf{P}_{h}^{l+/-}n_{h}^{l+/-} = n_{h}^{l+/-}$ on each $e_{h}^{l} \in \mathcal{E}_{h}^{l}$ and $\mathbf{H} \mathbf{P} = \mathbf{H}$, we have
\begin{align*} 
&\sum_{e_{h}^{l} \in \mathcal{E}_{h}^{l}}\int_{e_{h}^{l}}[u_{h}^{l}]\left(\left\{\nabla_{\Gamma}z_{h}^{l}; n \right\} - \left\{\mathbf{P}_{h}^{l}(\mathbf{I}-d\mathbf{H})\mathbf{P}\nabla_{\Gamma}z_{h}^{l}; n_{h}^{l}\right\} \delta_{e_{h}}^{-1}\right)\ \mbox{d}s \\
&= \sum_{e_{h}^{l} \in \mathcal{E}_{h}^{l}}\int_{e_{h}^{l}}[u_{h}^{l}]\left(\left\{\nabla_{\Gamma}z_{h}^{l}; n \right\} - \left\{\nabla_{\Gamma}z_{h}^{l}; \delta_{e_{h}}^{-1}\mathbf{P}(\mathbf{I}-d\mathbf{H})n_{h}^{l}\right\} \right)\ \mbox{d}s \\
&= \sum_{e_{h}^{l} \in \mathcal{E}_{h}^{l}}\int_{e_{h}^{l}}[u_{h}^{l}]\left(\left\{\nabla_{\Gamma}z_{h}^{l}; ( n - \delta_{e_{h}}^{-1}\mathbf{P} n_{h}^{l}) \right\} + \delta_{e_{h}}^{-1} d \left\{\nabla_{\Gamma}z_{h}^{l}; \mathbf{H} n_{h}^{l}\right\} \right)\ \mbox{d}s \\
&= \sum_{e_{h} \in \mathcal{E}_{h}}\int_{e_{h}}[u_{h}]\left(\left\{\mathbf{F}_{h} \nabla_{\Gamma_{h}}z_{h}; ( \delta_{e_{h}} n^{-l} - \mathbf{P}^{-l} n_{h}) \right\} + d \left\{\mathbf{F}_{h} \nabla_{\Gamma_{h}}z_{h}; \mathbf{H} n_{h}\right\} \right)\ \mbox{d}s_{h}.
\end{align*}
Making use of the above and writing all terms as element-wise computations, we derive the following residual equation:
\begin{align} \label{eq:apostderivation}
J(e_{h}) =  &\underbrace{\khsum{(f_{h}\delta_{h} + \Delta_{\Gamma_{h}}u_{h} - u_{h}\delta_{h})(z^{-l}-z_{h})}}_{I}  \underbrace{- \frac{1}{2} \sum_{K_{h} \in \mathcal{T}_{h}}\int_{\partial K_{h}}[\nabla_{\Gamma_{h}}u_{h} ; n_{h} ](z^{-l} - z_{h})\ \mbox{d}s_{h}}_{II} \notag \\
&\underbrace{- \khsum{\delta_{e_h}\mathbf{F}_{h}^{T}(\mathbf{P}^{-l}-\mathbf{A}_{h})\mathbf{F}_{h} \nabla_{\Gamma_{h}}u_{h} \cdot \nabla_{\Gamma_{h}}z^{-l}}}_{III} + \underbrace{\khsum{(1-\delta_{h})(u_{h} - f_{h})z_{h}}}_{IV} \notag \\ 
&+ \underbrace{\khsum{\mathbf{L}^{-l}(u_{h}^{l})\cdot \delta_{h}\mathbf{F}_{h}\nabla_{\Gamma_{h}}(z^{-l} - z_{h})}}_{V} \notag \\
&+ \underbrace{\sum_{K_{h} \in \mathcal{T}_{h}}\frac{1}{2}\int_{\partial K_{h}}[u_{h}]\left(\left\{\mathbf{F}_{h} \nabla_{\Gamma_{h}}z_{h}; ( \delta_{e_{h}} n^{-l} - \mathbf{P}^{-l} n_{h}) \right\} + d \left\{\mathbf{F}_{h} \nabla_{\Gamma_{h}}z_{h}; \mathbf{H} n_{h}\right\} \right)\ \mbox{d}s_{h}}_{VI} \notag \\ 
&+ \underbrace{A_{\Gamma}^{IP}(u_{h}^{l,\perp},z) - J(u_{h}^{l,\perp})}_{VII}.
\end{align}
\begin{remark}
With the exception of term $VII$, the residual equation (\ref{eq:apostderivation}) is fully computable and can be used to estimate an arbitrary bounded linear functional $J$ in $H^{1} + V_{h}^{l}$ of the error $e_{h}$ with high accuracy. In practice, we may deal with term $VII$ by bounding it as in (\ref{eq:boundForVII}) (making use of a suitable stability estimate for the dual solution $z$). The main drawback of performing error estimation based on (\ref{eq:apostderivation}) is that it requires approximating the weights $z^{-l}-z_{h}$, which typically involves finding an approximation of the solution $z$ to the dual problem (\ref{eq:DualProblem}) and thus requires an additional solve step at each iteration. From here on we will only focus on deriving estimates in the energy norm and will do so by bounding all of the terms in the residual equation (\ref{eq:apostderivation}), including the weights $z^{-l}-z_{h}$.                
\end{remark}

\section{A posteriori upper bound (reliability)}
In this section we derive a reliable estimator for the error in the energy norm.
\begin{theorem}\label{thm:Reliability}
Suppose that $\mathcal{T}_{h}$ is shape-regular and let
\begin{align}
\eta_{K_{h}} = h_{K_{h}} \norm{f_{h}\delta_{h} + \Delta_{\Gamma_{h}}u_{h} - u_{h}\delta_{h}}_{L^{2}(K_{h})} + h_{K_{h}}^{1/2} \norm{[\nabla_{\Gamma_{h}}u_{h} ; n_{h} ]}_{L^{2}(\partial K_{h})}
\end{align}
be the sum of the scaled element and jump residuals, then
\begin{align*} 
\norm{u - u_{h}^{l}}_{DG(\Gamma)} \leq  C \left( \sum_{K_{h} \in \mathcal{T}_{h}}\mathcal{R}_{K_{h}}^{2} + \mathcal{R}_{DG_{K_{h}}}^{2} + \mathcal{G}_{K_{h}}^{2} + \mathcal{G}_{DG_{K_{h}}}^{2} \right)^{\frac{1}{2}} 
\end{align*}
with
\begin{align}\label{eq:ResidualTerm}
\mathcal{R}_{K_{h}}^{2} := \norm{\mathbf{A}_{h}}_{l^{2},L^{\infty}(w_{K_{h}})}\eta_{K_{h}}^{2},   
\end{align}
\begin{align}\label{eq:DGTerm}
&\mathcal{R}_{DG_{K_{h}}}^{2} := \Big(1 + \norm{\mathbf{F}_{h}}_{l^{2},L^{\infty}(\partial K_{h})}^{2}\norm{\mathbf{P}_{h}(\mathbf{I} - d\mathbf{H})\mathbf{P}}_{l^{2},L^{\infty}(w_{K_{h}})}^2 \norm{\delta_{e_{h}}}_{L^{\infty}(\partial K_{h})}^2 \notag \\ 
&+ \norm{\mathbf{A}_{h}}_{l^{2},L^{\infty}(w_{K_{h}})} \norm{\delta_{h}}_{L^{\infty}(w_{K_{h}})}\norm{\mathbf{F}_{h}}_{l^{2},L^{\infty}(\partial K_{h})}^{2}\norm{\mathbf{P}_{h}(\mathbf{I} - d\mathbf{H})\mathbf{P}}_{l^{2},L^{\infty}(w_{K_{h}})}^2 \norm{\delta_{e_{h}}}_{L^{\infty}(\partial K_{h})}^2\Big)\norm{\sqrt{\beta_{e_{h}}}[u_{h}]}_{L^{2}(\partial K_{h})}^{2},
\end{align}
\begin{align}\label{eq:GeometricErrorTerm}
\mathcal{G}_{K_{h}}^{2} := &\norm{\mathbf{B}_{h}\nabla_{\Gamma_{h}}u_{h}}_{L^{2}(K_{h})}^{2} + \norm{(1-\delta_{h})(u_{h} - f_{h})}_{L^{2}(K_{h})}^{2}, 
\end{align}
\begin{align}\label{eq:DGGeometricErrorTerm}
\mathcal{G}_{DG_{K_{h}}}^{2} := \norm{\delta_{h}}_{L^{\infty}(w_{K_{h}})} h_{K_{h}}^{-2}\Bigg(&\norm{[u_{h}] \left\{ |  \left(\mathbf{F}_{h}\mathbf{P}_{h}\right)^{T}( \delta_{e_{h}} n^{-l} - \mathbf{P}^{-l} n_{h}) | \right\} }_{L^{2}(\partial K_{h})}^2 \notag
\\ 
&+ \norm{d [u_{h}] \left\{ | \left(\mathbf{F}_{h}\mathbf{P}_{h}\right)^{T} \mathbf{H} n_{h} | \right\}}_{L^{2}(\partial K_{h})}^2 \Bigg), 
\end{align}
where $C$ depends only on the shape regularity of the mesh and
$w_{K_{h}} = \bigcup_{p \in K_{h}}w_{p}$.
The operators $\mathbf{A}_h,\mathbf{B}_h$ are defined in 
(\ref{eq:Ah}) and (\ref{eq:Bh}), respectively.
\end{theorem}
The proof of Theorem \ref{thm:Reliability} will require the following norm equivalence result:
\begin{lemma}\label{thm:EquivalenceResult}
Assuming that $\omega_{e_{h}}$ is sufficiently large, the expression
\[v_{h}^{l,\perp} \rightarrow \left(\ehlsum{\beta_{e_{h}^{l}}|[v_{h}^{l, \perp}]|^{2}} \right)^{1/2}  \]
is a norm on $V_{h}^{l,\perp}$. This norm is equivalent to the norm $\norm{\cdot}_{DG(\Gamma)}$ and there is a constant $C_{\perp}$ such that 
\[\norm{v_{h}^{l, \perp}}_{DG(\Gamma)} \leq C_{\perp}\left(\ehlsum{\beta_{e_{h}^{l}}|[v_{h}^{l, \perp}]|^{2}} \right)^{1/2} \leq C_{\perp} \norm{v_{h}^{l, \perp}}_{DG(\Gamma)} \]
for all $v_{h}^{l, \perp} \in V_{h}^{l,\perp}$. The constant $C_{\perp}$ is independent of $h$ and depends on the shape-regularity of the mesh. 
\end{lemma}
\begin{proof}
See Theorem 2.2 in \cite{karakashian2003posteriori}.
\end{proof}

To prove Theorem \ref{thm:Reliability}, we begin by bounding term $I$ of (\ref{eq:apostderivation}). Let $z_{h} = I_{h} z^{-l}$, $R := f_{h}\delta_{h} + \Delta_{\Gamma_{h}}u_{h} - u_{h}\delta_{h}$, $r := [\nabla_{\Gamma_{h}}u_{h} ; n_{h} ]$. Recalling that $\{\varphi_{p}\}_{p \in \mathcal{N}}$ is a partition of unity, recalling (\ref{eq:ClementInterpolantProperty}) and applying (\ref{eq:PoicareInequalityL2PatchShapeRegular}), we then have
\begin{align} \label{eq:I}
I = \sum_{p \in \mathcal{N}} \int_{w_{p}} R(z^{-l} - z_{p}^{-l})\varphi_{p}\ \mbox{d}s_{h} \leq C \sum_{p \in \mathcal{N}} h_{p} \norm{\mathbf{A}_{h}}_{l^{2},L^{\infty}(w_{p})}^{\frac{1}{2}} \norm{R\varphi_{p}}_{L^{2}(w_{p})} \norm{\nabla_{\Gamma}z}_{L^{2}( w_{p}^l)}.   
\end{align} 
Next we turn to bounding term $II$. Applying (\ref{eq:PoicareInequalityL2edgeShapeRegular}), we find
\begin{align}\label{eq:II}
II = - \sum_{p \in \mathcal{N}}\sum_{\bar{e}_{h} \ni p}\int_{e_{h}} r(z^{-l} - z_{p}^{-l})\varphi_{p}\ \mbox{d}s_{h} \leq C \sum_{p \in \mathcal{N}}\sum_{\bar{e}_{h} \ni p}h_{p}^{\frac{1}{2}}\norm{\mathbf{A}_{h}}_{l^{2},L^{\infty}(w_{p})}^{\frac{1}{2}} \norm{r\varphi_{p}}_{L^{2}(e_{h})}\norm{\nabla_{\Gamma}z}_{L^{2}(w_{p}^l)}. 
\end{align}
Let
\[ \eta_{p} = h_{p} \norm{R\varphi_{p}}_{L^{2}(w_{p})} + \sum_{\bar{e}_{h} \ni p}h_{p}^{\frac{1}{2}} \norm{r\varphi_{p}}_{L^{2}(e_{h})}. \]
Combining (\ref{eq:I}) and (\ref{eq:II}) and noting that each element $H_{h}$ has only three nodes, we thus find that

\begin{align}\label{eq:I+II}
I + II &\leq C \sum_{p \in \mathcal{N}} \norm{\mathbf{A}_{h}}_{l^{2},L^{\infty}(w_{p})}^{\frac{1}{2}} \eta_{p} \norm{\nabla_{\Gamma} z}_{L^{2}(w_{p}^l)} \leq C \left( \sum_{p \in \mathcal{N}} \norm{\mathbf{A}_{h}}_{l^{2},L^{\infty}(w_{p})} \eta_{p}^{2} \right)^{\frac{1}{2}} \norm{z}_{H^{1}(\Gamma)}     
\end{align}
where $C$ does not depend on $\mathcal{T}_{h}$ or any other essential quantities.

In order to bound term $III$ in (\ref{eq:apostderivation}) we first surface lift the integral back to $\Gamma$, and making use of (\ref{eq:Gamma2GammahGradient}) we get 
\begin{align*}
III = - \sum_{K_{h}^{l} \in \mathcal{T}_{h}^{l}}\int_{K_{h}^{l}}(\mathbf{P}-\mathbf{A}_{h}^{l})\nabla_{\Gamma}u_{h}^{l}\nabla_{\Gamma} z\ \mbox{d}s \leq \left(\sum_{K_{h} \in \mathcal{T}_{h}} \norm{\mathbf{B}_{h}\nabla_{\Gamma_{h}}u_{h}}^{2}_{L^{2}(K_{h})}\right)^{1/2} \norm{z}_{H^{1}(\Gamma)}.
\end{align*}
Next we bound term $IV$. First we note that, for $p \in \mathcal{N}$ and with $z^{-l}_{p}$ defined as in (\ref{eq:ClementInterpolantWeights}), we have
\begin{align*}
 \norm{\sqrt{\varphi_{p}}z_{p}^{-l}}_{L^{2}(w_{p})} &=\sqrt{\int_{w_{p}}\varphi_{p}\ \mbox{d}s_{h}} \frac{1}{\int_{w_{p}}\varphi_{p}\ \mbox{d}s_{h}} \left|\int_{w_{p}} \varphi_{p}z^{-l}_{p}\ \mbox{d}s_{h} \right| \leq \norm{\sqrt{\varphi_{p}}z^{-l}}_{L^{2}(w_{p})}. 
\end{align*}
Making use of the above, we have the following:
\begin{align*}
 IV &= \khsum{(1-\delta_{h})(u_{h} - f_{h})z_{h}} \leq \sum_{p \in \mathcal{N}} \norm{\sqrt{\varphi_{p}}(1-\delta_{h})(u_{h} - f_{h})}_{L^{2}(w_{p})}\norm{\sqrt{\varphi_{p}}z^{-l}_{p}}_{L^{2}(w_{p})}\\ 
&\leq \sum_{p \in \mathcal{N}} \norm{\delta_{h}^{-1}}_{L^{\infty}(w_{p})}^{1/2} \norm{\sqrt{\varphi_{p}}(1-\delta_{h})(u_{h} - f_{h})}_{L^{2}(w_{p})}\norm{\sqrt{\varphi_{p}^{l}} z}_{L^{2}(w_{p}^{l})} \\
&\leq \left(\sum_{p \in \mathcal{N}} \norm{\sqrt{\varphi_{p}}(1-\delta_{h})(u_{h} - f_{h})}_{L^{2}(w_{p})}^{2}\right)^{1/2} \norm{z}_{H^{1}(\Gamma)}.
\end{align*}
Making use of Remark \ref{LiftingEstimateRemark2}, we may bound term $V$ in the following way:
\begin{align*}
V &= \sum_{p \in \mathcal{N}}\int_{w_{p}}\mathbf{L}^{-l}(u_{h}^{l}) \cdot \delta_{h}\mathbf{F}_{h}\nabla_{\Gamma_{h}}((z^{-l} - z^{-l}_{p})\varphi_{p})\ \mbox{d}\sigma_{h}\\ 
&= \sum_{p \in \mathcal{N}} \int_{w_{p}}\mathbf{L}^{-l}(u_{h}^{l}) \cdot \left( \delta_{h}\mathbf{F}_{h} \nabla_{\Gamma_{h}}z^{-l}\varphi_{p} + \delta_{h}\mathbf{F}_{h}(z^{-l} - z^{-l}_{p})\nabla_{\Gamma_{h}}\varphi_{p}   \right)\ \mbox{d}\sigma_{h}\\ 
&= \sum_{p \in \mathcal{N}} \int_{w_{p}^l}\mathbf{L}(u_{h}^{l}) \cdot \nabla_{\Gamma}z \varphi_{p}^{l}\ \mbox{d}\sigma + \int_{w_{p}^{l}}\mathbf{L}(u_{h}^{l})\cdot \nabla_{\Gamma}\varphi_{p}^l (z - z_{p})\ \mbox{d}\sigma \\
&\leq \sum_{p \in \mathcal{N}} \norm{\mathbf{L}(u_{h}^{l})\sqrt{\varphi_{p}^{l}}}_{L^{2}(w_{p}^l)} \norm{\nabla_{\Gamma}z\sqrt{\varphi_{p}^{l}}}_{L^{2}(w_{p}^l)} + \sum_{p \in \mathcal{N}} \norm{\mathbf{L}(u_{h}^{l})\cdot \nabla \varphi_{p}^{l}}_{L^{2}(w_{p}^{l})} \norm{\delta_{h}}_{L^{\infty}(w_{p})}^{1/2}\norm{z - z_{p}}_{L^{2}(w_{p})} \\
&\leq \left(\sum_{p \in \mathcal{N}} \int_{w_{p}^{l}}\left(\sum_{e_{h}^{l} \subset \bar{w}_{p}^{l}}\mathbf{L}_{e_{h}^l}(u_{h}^{l})\right)^{2}\varphi_{p}^{l}\ \mbox{d}\sigma \right)^{1/2} \left(\sum_{p \in \mathcal{N}}\norm{\nabla_{\Gamma}z\sqrt{\varphi_{p}^{l}}}_{L^{2}(w_{p}^l)}^{2}\right)^{1/2} \\
&+ C \sqrt{2} \sum_{p \in \mathcal{N}} h_{p}^{-1} \norm{\mathbf{L}(u_{h}^{l})}_{L^{2}(w_{p}^{l})} \norm{\delta_{h}}_{L^{\infty}(w_{p})}^{1/2} h_{p} \norm{\mathbf{A}_{h}}_{l^{2},L^{\infty}(w_{p})}^{\frac{1}{2}} \norm{\nabla_{\Gamma}z}_{L^{2}(w_{p}^l)} \\
&\leq C \left(\sum_{p \in \mathcal{N}} \int_{w_{p}^{l}}\sum_{e_{h}^{l} \subset \bar{w}_{p}^{l}}\mathbf{L}_{e_{h}^l}^2(u_{h}^{l})\varphi_{p}^{l}\ \mbox{d}\sigma \right)^{1/2} \norm{z}_{H^{1}(\Gamma)}\\ 
&+ C \sqrt{2}\left( \sum_{p \in \mathcal{N}} \norm{\mathbf{A}_{h}}_{l^{2},L^{\infty}(w_{p})} \norm{\delta_{h}}_{L^{\infty}(w_{p})} \int_{w_{p}^{l}}\sum_{e_{h}^{l} \subset \bar{w}_{p}^{l}}\mathbf{L}_{e_{h}^l}^2(u_{h}^{l})\ \mbox{d}\sigma  \right)^{1/2} \norm{z}_{H^{1}(\Gamma)}.
\end{align*}
Using again Remark \ref{LiftingEstimateRemark2} and the DG lifting estimate in Lemma \ref{LiftingEstimate}, we have that
\begin{align*} 
&\sum_{p \in \mathcal{N}} \int_{w_{p}^{l}}\sum_{e_{h}^{l} \subset \bar{w}_{p}^{l}}\mathbf{L}_{e_{h}^l}^2(u_{h}^{l})\varphi_{p}^{l}\ \mbox{d}\sigma \leq \sum_{e_{h}^{l} \in \mathcal{E}_{h}^{l}} \int_{\Gamma}\mathbf{L}_{e_{h}^l}^{2}(u_{h}^{l})\ \mbox{d}\sigma = \sum_{e_{h}^{l} \in \mathcal{E}_{h}^{l}} \int_{w_{e_{h}}^{l}}\mathbf{L}_{e_{h}^l}^{2}(u_{h}^{l})\ \mbox{d}\sigma \\ 
&\leq C_{L} \sum_{e_{h} \in \mathcal{E}_{h}}\norm{\mathbf{F}_{h}}_{l^{2},L^{\infty}(e_{h})}^{2}\norm{\mathbf{P}_{h}(\mathbf{I} - d\mathbf{H})\mathbf{P}}_{l^{2},L^{\infty}(w_{e_{h}})}^2 \norm{\delta_{e_{h}}}_{L^{\infty}(e_{h})}^2\norm{\sqrt{\beta_{e_{h}}}[u_{h}]}_{L^{2}(e_{h})}^{2}. 
\end{align*}
Similarly, we have
\begin{align*}
&\sum_{p \in \mathcal{N}} \int_{w_{p}^{l}}\sum_{e_{h}^{l} \subset \bar{w}_{p}^{l}}\mathbf{L}_{e_{h}^l}^2(u_{h}^{l})\ \mbox{d}\sigma \leq C \sum_{K_{h}^{l} \in \mathcal{T}_{h}^{l}} \int_{K_{h}^{l}} \sum_{e_{h}^{l} \subset \partial K_{h}^{l}}\mathbf{L}_{e_{h}^l}^2(u_{h}^{l})\ \mbox{d}\sigma \\
&\leq C \sum_{K_{h}^{l} \in \mathcal{T}_{h}^{l}} \sum_{e_{h}^{l} \subset \partial K_{h}^{l}} \int_{\Gamma} \mathbf{L}_{e_{h}^l}^2(u_{h}^{l})\ \mbox{d}\sigma
= C \sum_{K_{h}^{l} \in \mathcal{T}_{h}^{l}} \sum_{e_{h}^{l} \subset \partial K_{h}^{l}} \int_{w_{e_{h}^l}} \mathbf{L}_{e_{h}^l}^2(u_{h}^{l})\ \mbox{d}\sigma \\
&\leq  C C_{L} \sum_{K_{h} \in \mathcal{T}_{h}} \sum_{e_{h} \subset \partial K_{h}} \norm{\mathbf{F}_{h}}_{l^{2},L^{\infty}(e_{h})}^{2}\norm{\mathbf{P}_{h}(\mathbf{I} - d\mathbf{H})\mathbf{P}}_{l^{2},L^{\infty}(w_{e_{h}})}^2 \norm{\delta_{e_{h}}}_{L^{\infty}(e_{h})}^2\norm{\sqrt{\beta_{e_{h}}}[u_{h}]}_{L^{2}(e_{h})}^{2}. 
\end{align*}

For term $VI$ we have the following,
\begin{align*} 
VI &= \sum_{e_{h} \in \mathcal{E}_{h}}\int_{e_{h}}[u_{h}]\left(\left\{\mathbf{F}_{h} \nabla_{\Gamma_{h}}z_{h}; ( \delta_{e_{h}} n^{-l} - \mathbf{P}^{-l} n_{h}) \right\} + d \left\{\mathbf{F}_{h} \nabla_{\Gamma_{h}}z_{h}; \mathbf{H} n_{h}\right\} \right)\ \mbox{d}s_{h} \\
&= \sum_{p \in \mathcal{N}}\sum_{\bar{e}_{h} \ni p}\int_{e_{h}}z_{p}^{-l}[u_{h}]\left\{\mathbf{F}_{h} \nabla_{\Gamma_{h}}\varphi_{p}; ( \delta_{e_{h}} n^{-l} - \mathbf{P}^{-l} n_{h}) \right\} + z_{p}^{-l}[u_{h}]d  \left\{\mathbf{F}_{h} \nabla_{\Gamma_{h}}\varphi_{p}; \mathbf{H} n_{h}\right\}\ \mbox{d}s_{h}.
\end{align*}
Making use of (\ref{eq:ClementInterpolantWeightsProperty}) and recalling that $\nabla_{\Gamma_{h}}\varphi_{p} = \mathbf{P}_{h}\nabla \varphi_{p}$, the first term of $VI$ becomes
\begin{align*}
&\sum_{p \in \mathcal{N}} \sum_{\bar{e}_{h} \ni p} \norm{[u_{h}] \left\{\mathbf{F}_{h} \nabla_{\Gamma_{h}}\varphi_{p}; ( \delta_{e_{h}} n^{-l} - \mathbf{P}^{-l} n_{h}) \right\} }_{L^{2}(e_{h})} \norm{z^{-l}_{p}}_{L^{2}(e_{h})} \\
&\leq \sqrt{\frac{3}{2}}\sum_{p \in \mathcal{N}}\norm{\sqrt{\delta_{h}}}_{L^{\infty}(w_{p})} \sum_{\bar{e}_{h} \ni p} \norm{[u_{h}] \left\{\nabla \varphi_{p}; \left(\mathbf{F}_{h}\mathbf{P}_{h}\right)^{T}( \delta_{e_{h}} n^{-l} - \mathbf{P}^{-l} n_{h}) \right\} }_{L^{2}(e_{h})} \norm{z}_{L^{2}(w_{p}^l)} \\
&\leq \sqrt{3}\sum_{p \in \mathcal{N}}\norm{\sqrt{\delta_{h}}}_{L^{\infty}(w_{p})} \sum_{\bar{e}_{h} \ni p} h_{e_{h}}^{-1}\norm{ [u_{h}] \left\{ |  \left(\mathbf{F}_{h}\mathbf{P}_{h}\right)^{T}( \delta_{e_{h}} n^{-l} - \mathbf{P}^{-l} n_{h}) | \right\} }_{L^{2}(e_{h})} \norm{z}_{L^{2}(w_{p}^{l})} \\
&\leq C \sqrt{3}\left( \sum_{p \in \mathcal{N}}\norm{\delta_{h}}_{L^{\infty}(w_{p})} \sum_{\bar{e}_{h} \ni p} h_{e_{h}}^{-2}\norm{ [u_{h}] \left\{ |  \left(\mathbf{F}_{h}\mathbf{P}_{h}\right)^{T}( \delta_{e_{h}} n^{-l} - \mathbf{P}^{-l} n_{h}) | \right\}}_{L^{2}(e_{h})}^{2} \right)^{1/2} \norm{z}_{H^{1}(\Gamma)}.
\end{align*}
Similarly, for the second term of $VI$, we get
\begin{align*}
C \sqrt{3}\left( \sum_{p \in \mathcal{N}}\norm{\delta_{h}}_{L^{\infty}(w_{p})} \sum_{\bar{e}_{h} \ni p} h_{e_{h}}^{-2}\norm{d [u_{h}]  \left\{| \left(\mathbf{F}_{h}\mathbf{P}_{h}\right)^{T} \mathbf{H} n_{h} |\right\}}_{L^{2}(e_{h})}^{2} \right)^{1/2} \norm{z}_{H^{1}(\Gamma)}. 
\end{align*}

To bound the final term $VII$ in our residual equation, we first prescribe the functional $J$ as follows:  
\[J(v) = \norm{\mbox{e}_{h}}_{DG(\Gamma)}^{-1}\left(\sum_{K_{h}^{l} \in \mathcal{T}_{h}^{l}}(\mbox{e}_{h},v)_{H^{1}(K_{h}^{l})} + \sum_{e_{h}^{l} \in \mathcal{E}_{h}^{l}}h_{e_{h}^{l}}^{-1}([\mbox{e}_{h}],[v])_{L^{2}(e_{h}^{l})}    \right)  \]
which is in fact a functional on $H^{1}(\Gamma) + V_{h}^{l}$. Note that $J(\mbox{e}_{h}) = \norm{\mbox{e}_{h}}_{DG(\Gamma)}$. For such a functional, the solution $z$ of the dual problem (\ref{eq:DualProblem}) satisfies
\[\norm{z}_{H^{1}(\Gamma)}^{2} \leq J(z) \leq \norm{z}_{DG(\Gamma)} = \norm{z}_{H^{1}(\Gamma)},  \]
where we have used that $\mathbf{L}(z) = 0$ and $[z] = 0$ since $z \in H^{1}(\Gamma)$. Hence $\norm{z}_{H^{1}(\Gamma)} \leq 1$. Making use of this stability estimate, the lifting estimate given in Lemma \ref{LiftingEstimate} and the norm equivalence result in Lemma \ref{thm:EquivalenceResult}, we have
\begin{align}\label{eq:boundForVII}
&VII := A_{\Gamma}^{IP}(u_{h}^{l,\perp},z) - J(u_{h}^{l,\perp}) \leq C \Bigg( \sum_{K_{h} \in \mathcal{T}_{h}} \norm{\sqrt{\beta_{e_{h}}}[u_{h}]}_{L^{2}(\partial K_{h})}^{2}\notag \\ 
&+ \norm{\mathbf{F}_{h}}_{l^{2},L^{\infty}(\partial K_{h})}^{2}\norm{\mathbf{P}_{h}(\mathbf{I} - d\mathbf{H})\mathbf{P}}_{l^{2},L^{\infty}(w_{K_{h}})}^2 \norm{\delta_{e_{h}}}_{L^{\infty}(\partial K_{h})}^2 \norm{\sqrt{\beta_{e_{h}}}[u_{h}]}_{L^{2}(\partial K_{h})}^{2} \Bigg)^{1/2}. 
\end{align}
Combining all of the estimates in this section and writing them in terms of element-wise computations completes the proof of Theorem \ref{thm:Reliability}.

\begin{lemma} \label{Gamma2GammahSmall}
Let $\Gamma$ be an oriented $C^{2}$ surface in $\mathbb{R}^{3}$ and $\Gamma_{h}$ its linear interpolation with outward unit normal $\nu_{h}$. Then we have
\begin{align*}
\norm{d}_{L^{\infty}(\Gamma_{h})} \leq Ch^{2},\ \norm{1-\delta_{h}}_{L^{\infty}(\Gamma_{h})} \leq &Ch^{2},\ \norm{\mathbf{P}^{-l}-\mathbf{A}_{h}}_{l^{2},L^{\infty}(\Gamma_{h})} \leq Ch^{2},\ \\ 
\norm{\mathbf{B}_{h}}_{l^{2},L^{\infty}(\Gamma_{h})} \leq Ch^{2},\ \norm{1-\delta_{e_{h}}}_{L^{\infty}(\Gamma_{h})} \leq &Ch^{2}\ \mbox{and}\ \norm{\left\{|n^{-l}- \mathbf{P}^{-l} n_{h}|\right\}}_{L^{\infty}(\mathcal{E}_{h})} \leq C h^{2}.
\end{align*}
\end{lemma}
\begin{proof}
See \cite{dziuk1988finite} for the first three estimates and Lemma 3.3 and 3.5 in \cite{GieMue_prep} for the fifth and sixth estimate, respectively. The fourth estimate follows straightforwardly from the third estimate. 

\end{proof}
\begin{remark}
The geometric estimates in Lemma  \ref{Gamma2GammahSmall} make it clear that, if $\Gamma$ is sufficiently smooth, $\mathcal{G}_{K_{h}}$ and $\mathcal{G}_{DG_{K_{h}}}$ are of higher order compared to $\mathcal{R}_{K_{h}}$ and $\mathcal{R}_{DG_{K_{h}}}$ i.e.
\begin{align*}
\left(\sum_{K_{h} \in \mathcal{T}_{h}}\mathcal{R}_{K_{h}}^{2} + \mathcal{R}_{DG_{K_{h}}}^{2}\right)^{1/2} \leq Ch\  \  \mbox{and}\  \ \left(\sum_{K_{h} \in \mathcal{T}_{h}}\mathcal{G}_{K_{h}}^{2} + \mathcal{G}_{DG_{K_{h}}}^{2}\right)^{1/2} \leq Ch^{2}.   
\end{align*}  
\end{remark}

\section{A posteriori lower bound (efficiency)}
We now show that the estimator in Theorem \ref{thm:Reliability} is efficient up to higher-order terms.
\begin{theorem}
Suppose that $\mathcal{T}_{h}$ is shape-regular. As before, let $R := f_{h}\delta_{h} + \Delta_{\Gamma_{h}}u_{h} - u_{h}\delta_{h}$ and $r := [\nabla_{\Gamma_{h}}u_{h} ; n_{h} ]$. Then for each $K_{h} \in \mathcal{T}_{h}$ we have
\begin{align*}
\eta_{K_{h}} + \norm{\sqrt{\beta_{e_{h}}}[u_{h}]}_{L^{2}(\partial K_{h})} &\leq C \max \left\{1,  \norm{\mathbf{A}_{h}}_{l^{2},L^{\infty}(w_{K_{h}})}^{1/2}\right\} \left(\norm{u - u_{h}^{l}}_{DG(w_{K_{h}}^l)} + \norm{\mathbf{B}_{h}\nabla_{\Gamma_{h}}u_{h}}_{L^{2}(w_{K_{h}})} \right)\\ 
&+ C h_{K_{h}} \norm{R - \bar{R}}_{L^{2}(w_{K_{h}})} + C h_{K_{h}}^{1/2} \norm{r - \bar{r}}_{L^{2}(\partial K_{h})}.
\end{align*}
where $\eta_{K_{h}}$ is given in Theorem \ref{thm:Reliability}. Here $C$ depends on the number of elements in $w_{K_{h}}$, the minimum angle of the elements in $w_{K_{h}}$ and on the upper bound for the penalty values $\omega_{e_h}$. $\bar{R}$ and $\bar{r}$ are respectively piecewise linear approximations of $R$ and $r$.
\end{theorem}
\begin{proof}
The proof will follow the bubble function approach considered in \cite{verfurth1989posteriori}, which was then straightforwardly applied to the DG framework in \cite{schotzau2009robust}. First we bound the element residual $\norm{R}_{L^{2}(K_{h})}$. Let $p \in \mathcal{N}$ and $K_{h} \subset w_{p}$. Letting $p_{i}$, $1 \leq i \leq 3$ be the nodes of $K_{h}$, we define the bubble function $\phi_{K_{h}} = \prod_{i=1}^{3}\varphi_{p_{i}}$. Integrating by parts on $K_{h}$, lifting the resulting integral onto $K_{h}^{l}$, making use of the fact that the exact solution satisfies $(f + \Delta_{\Gamma}u - u)|_{K_{h}^{l}} = 0$ and integrating by parts on $K_{h}^{l}$, we get
\begin{align*}
\int_{K_{h}} R \bar{R}\phi_{K_{h}}\ \mbox{d}\sigma_{h} &= \int_{K_{h}}f_{h}\delta_{h}\bar{R}\phi_{K_{h}} + \nabla_{\Gamma_{h}}u_{h} \cdot \nabla_{\Gamma_{h}}(\bar{R}\phi_{K_{h}}) - u_{h}\delta_{h} \bar{R}\phi_{K_{h}}\ \mbox{d}\sigma_{h} \\
&= \int_{K_{h}^{l}} f \bar{R}^{l}\phi_{K_{h}}^{l} + \mathbf{A}_{h}^{l}\nabla_{\Gamma}u_{h}^{l}\cdot \nabla_{\Gamma}(\bar{R}^{l}\phi_{K_{h}}^{l}) - u_{h}^{l} \bar{R}^{l}\phi_{K_{h}}^{l}\ \mbox{d}\sigma \\
&= \int_{K_{h}^{l}}\nabla_{\Gamma}(u-u_{h}^{l})\cdot \nabla_{\Gamma}(\bar{R}^{l}\phi_{K_{h}}^{l})\ \mbox{d}\sigma + \int_{K_{h}^{l}}(u-u_{h}^{l})\bar{R}^{l}\phi_{K_{h}}^{l}\ \mbox{d}\sigma \\
&+ \int_{K_{h}^{l}}(\mathbf{P}-\mathbf{A}_{h}^{l})\nabla_{\Gamma}u_{h}^{l} \cdot \nabla_{\Gamma}(\bar{R}^{l}\phi_{K_{h}}^{l})\ \mbox{d}\sigma.  
\end{align*}
Note that we have used the fact that $\phi_{K_{h}} = 0$ on $\partial K_{h}$ so that all boundary terms resulting from the integration by parts vanish. We then have
\begin{align*}
\int_{K_{h}} R \bar{R}\phi_{K_{h}}\ \mbox{d}\sigma_{h} &\leq C\left( \norm{u-u_{h}^{l}}_{DG(K_{h}^{l})} + \norm{(\mathbf{P}-\mathbf{A}_{h}^{l})\nabla_{\Gamma}u_{h}^{l}}_{L^{2}(K_{h}^{l})}\right) \norm{\nabla_{\Gamma}(\bar{R}^{l}\phi_{K_{h}}^{l})}_{L^{2}(K_{h}^{l})} \\
&\leq C\left( \norm{u-u_{h}^{l}}_{DG(K_{h}^{l})} + \norm{\mathbf{B}_{h}\nabla_{\Gamma}u_{h}}_{L^{2}(K_{h})}\right) \norm{\mathbf{A}_{h}}_{L^{\infty}(K_{h})}^{1/2}\norm{\nabla_{\Gamma_{h}}(\bar{R}\phi_{K_{h}})}_{L^{2}(K_{h})}
\end{align*}
where we have used Poincare's inequality. Since $\bar{R}\phi_{K_{h}}$ is a polynomial, it satisfies the inverse inequality
\[ \norm{\nabla_{\Gamma_{h}}(\bar{R}\phi_{K_{h}})}_{L^{2}(K_{h})} \leq C h_{K_{h}}^{-1} \norm{\bar{R}}_{L^{2}(K_{h})}  \]
where $C$ depends only on the shape-regularity of $K_{h}$. Applying this inverse inequality, we get
\begin{align*}
\int_{K_{h}} R \bar{R}\phi_{K_{h}}\ \mbox{d}\sigma_{h} \leq C h_{K_{h}}^{-1} \norm{\mathbf{A}_{h}}_{L^{\infty}(K_{h})}^{1/2}\left( \norm{u-u_{h}^{l}}_{DG(K_{h}^{l})} + \norm{\mathbf{B}_{h}\nabla_{\Gamma}u_{h}}_{L^{2}(K_{h})}\right)\norm{\bar{R}}_{L^{2}(K_{h})}. 
\end{align*}
Applying Theorem 2.2 in \cite{ainsworth2011posteriori}, we have
\begin{align*}
\norm{\bar{R}}_{L^{2}(K_{h})}^{2} &\leq  C\norm{\sqrt{\phi_{K_{h}}}\bar{R}}_{L^{2}(K_{h})}^{2} \\
&\leq C\left( \int_{K_{h}} R \bar{R}\phi_{K_{h}}\ \mbox{d}\sigma_{h} + \int_{K_{h}}\bar{R}(\bar{R} - R)\phi_{K_{h}}\ \mbox{d}\sigma_{h} \right) \\
&\leq C\left( \int_{K_{h}} R \bar{R}\phi_{K_{h}}\ \mbox{d}\sigma_{h} + \norm{R - \bar{R}}_{L^{2}(K_{h})}\norm{\bar{R}\phi_{K_{h}}}_{L^{2}(K_{h})} \right) \\
&\leq C\left( \int_{K_{h}} R \bar{R}\phi_{K_{h}}\ \mbox{d}\sigma_{h} + \norm{R - \bar{R}}_{L^{2}(K_{h})}\norm{\bar{R}}_{L^{2}(K_{h})} \right).
\end{align*}
Combining this with the previous inequality, we get
\begin{align*}
 \norm{\bar{R}}_{L^{2}(K_{h})}^{2} &\leq \left(\norm{R - \bar{R}}_{L^{2}(K_{h})} + C h_{K_{h}}^{-1} \norm{\mathbf{A}_{h}}_{L^{\infty}(K_{h})}^{1/2}\left( \norm{u-u_{h}^{l}}_{DG(K_{h}^{l})} + \norm{\mathbf{B}_{h}\nabla_{\Gamma}u_{h}}_{L^{2}(K_{h})}\right) \right)\norm{\bar{R}}_{L^{2}(K_{h})}. 
\end{align*}
Dividing both sides by $\norm{\bar{R}}_{L^{2}(K_{h})}$ and making use of the triangle inequality, we obtain
\begin{align*}
h_{K_{h}}\norm{R}_{L^{2}(K_{h})} \leq C \left( \norm{\mathbf{A}_{h}}_{L^{\infty}(K_{h})}^{1/2}\left( \norm{u-u_{h}^{l}}_{DG(K_{h}^{l})} + \norm{\mathbf{B}_{h}\nabla_{\Gamma}u_{h}}_{L^{2}(K_{h})}\right) + h_{K_{h}} \norm{R - \bar{R}}_{L^{2}(K_{h})} \right). 
\end{align*}
Next we bound the jump residual $\norm{r}_{L^{2}(\partial K_{h})}$. Let $e_{h}$ be an edge which is shared by elements $K_{h}^{1} = K_{h}$ and $K_{h}^{2}$ and whose closure contains the nodes $p_{1}$ and $p_{2}$. Let $\lambda_{i,j}$, $i,j=1,2,$ be the barycentric coordinate on triangle $i$ corresponding to vertex $p_{j}$, and define $\phi_{e_{h}}|_{K_{h}^{i}} = \lambda_{i,1} \lambda_{i,2}$. Thus $\phi_{e_{h}} \in H^{1}_{0}(K_{h}^{1} \cup K_{h}^{2})$, and $\phi_{e_{h}} > 0$ on $e_{h}$. Finally let $w_{e_{h}} = K_{h}^{1} \cup K_{h}^{2}$. Applying similar arguments as for the element residual $\norm{R}_{L^{2}(K_{h})}$, we have
\begin{align*}
\int_{e_{h}} r \bar{r}\phi_{e_{h}}\ \mbox{d}s_{h} &= \int_{w_{e_{h}}} \Delta_{\Gamma_{h}}u_{h} \bar{r} \phi_{e_{h}} + \nabla_{\Gamma_{h}}u_{h} \cdot \nabla_{\Gamma_{h}}(\bar{r}\phi_{e_{h}})\ \mbox{d}\sigma_{h} \\
&=  \int_{w_{e_{h}}} R \bar{r} \phi_{e_{h}}\ \mbox{d}\sigma_{h} + \int_{w_{e_{h}}^{l}} \mathbf{A}_{h}^{l}\nabla_{\Gamma}u_{h}^{l} \cdot \nabla_{\Gamma}(\bar{r}^{l}\phi_{e_{h}}^{l})\ \mbox{d}\sigma + \int_{w_{e_{h}}^{l}}(u_{h}^{l}-f)\bar{r}^{l}\phi_{e_{h}}^{l}\ \mbox{d}\sigma \\
&= \int_{w_{e_{h}}} R \bar{r} \phi_{e_{h}}\ \mbox{d}\sigma_{h} + \int_{w_{e_{h}}^{l}} \nabla_{\Gamma}(u_{h}^{l}-u) \cdot \nabla_{\Gamma}(\bar{r}^{l}\phi_{e_{h}}^{l})\ \mbox{d}\sigma + \int_{w_{e_{h}}^{l}}(u_{h}^{l}-u)\bar{r}^{l}\phi_{e_{h}}^{l}\ \mbox{d}\sigma \\ &+ \int_{w_{e_{h}}^{l}} (\mathbf{A}_{h}^{l}-\mathbf{P})\nabla_{\Gamma}u_{h}^{l} \cdot \nabla_{\Gamma}(\bar{r}^{l}\phi_{e_{h}}^{l})\ \mbox{d}\sigma  
\end{align*}
where again we have used the fact that $\phi_{e_{h}} = 0$ on $\partial w_{e_{h}}$ so that all boundary terms resulting from the integration by parts vanish. We now proceed to bounding the terms as done previously to obtain
\begin{align*}
\int_{e_{h}} r \bar{r}\phi_{e_{h}}\ \mbox{d}s_{h} &\leq C \Bigg(\norm{R}_{L^{2}(w_{e_{h}})} \norm{\bar{r} \phi_{e_{h}}}_{L^{2}(w_{e_{h}})} \\
&+ \left( \norm{u-u_{h}^{l}}_{DG(w_{e_{h}}^{l})} + \norm{\mathbf{B}_{h} \nabla_{\Gamma_{h}}u_{h}}_{L^{2}(w_{e_{h}})}     \right)\norm{\mathbf{A}_{h}}_{L^{\infty}(w_{e_{h}})}^{1/2}\norm{\nabla_{\Gamma_{h}}(\bar{r}\phi_{e_{h}})}_{L^{2}(w_{e_{h}})}\Bigg)
\end{align*}
where again the constant $C$ depends only on the shape regularity of the mesh. Since $\bar{r}\phi_{e_{h}}$ is a polynomial, it satisfies the inverse inequalities
\[\norm{\bar{r}\phi_{e_{h}}}_{L^{2}(w_{e_{h}})} \leq C h_{K_{h}}^{1/2} \norm{\bar{r}}_{L^{2}(e_{h})}\  \  \ , \  \  \ \norm{\nabla_{\Gamma_{h}}(\bar{r}\phi_{e_{h}})}_{L^{2}(w_{e_{h}})} \leq C h_{K_{h}}^{-1/2} \norm{\bar{r}}_{L^{2}(e_{h})}.    \]
Applying these inverse inequalities, we get
\begin{align*}
\int_{e_{h}} r \bar{r}\phi_{e_{h}}\ \mbox{d}s_{h} &\leq C \Bigg( h_{K_{h}}^{-1/2} \norm{\mathbf{A}_{h}}_{L^{\infty}(w_{e_{h}})}^{1/2}\left( \norm{u-u_{h}^{l}}_{DG(w_{e_{h}}^{l})} + \norm{\mathbf{B}_{h}\nabla_{\Gamma}u_{h}}_{L^{2}(w_{e_{h}})}\right) \\
&+ h_{K_{h}}^{1/2}\norm{R}_{L^{2}(w_{e_{h}})} \Bigg)\norm{\bar{r}}_{L^{2}(e_{h})}.
\end{align*}
Applying Theorem 2.4 in \cite{ainsworth2011posteriori}, we have
\begin{align*}
\norm{\bar{r}}_{L^{2}(e_{h})}^{2} &\leq  C \norm{\sqrt{\phi_{e_{h}}}\bar{r}}_{L^{2}(e_{h})}^{2} \\
&\leq C \left(\int_{e_{h}} r \bar{r}\phi_{e_{h}}\ \mbox{d}\sigma_{h} + \norm{r - \bar{r}}_{L^{2}(e_{h})} \norm{\bar{r}\phi_{e_{h}}}_{L^{2}(e_{h})}\right) \\
&\leq C \left(\int_{e_{h}} r \bar{r}\phi_{e_{h}}\ \mbox{d}\sigma_{h} + \norm{r - \bar{r}}_{L^{2}(e_{h})} \norm{\bar{r}}_{L^{2}(e_{h})}\right).
\end{align*}
Combining this with the previous inequality, we get
\begin{align*}
 \norm{\bar{r}}_{L^{2}(e_{h})}^{2} &\leq C \Bigg( \norm{r - \bar{r}}_{L^{2}(e_{h})} + h_{K_{h}}^{-1/2} \norm{\mathbf{A}_{h}}_{L^{\infty}(w_{e_{h}})}^{1/2}\left( \norm{u-u_{h}^{l}}_{DG(w_{e_{h}}^{l})} + \norm{\mathbf{B}_{h}\nabla_{\Gamma}u_{h}}_{L^{2}(w_{e_{h}})}\right) \\
&+ h_{K_{h}}^{1/2}\norm{R}_{L^{2}(w_{e_{h}})} \Bigg)\norm{\bar{r}}_{L^{2}(e_{h})}. 
\end{align*}
Dividing both sides by $\norm{\bar{r}}_{L^{2}(e_{h})}$ and making use of the triangle inequality, we obtain
\begin{align*}
h_{K_{h}}^{1/2}\norm{r}_{L^{2}(e_{h})} &\leq C \Bigg( \norm{\mathbf{A}_{h}}_{L^{\infty}(K_{h})}^{1/2}\left( \norm{u-u_{h}^{l}}_{DG(K_{h}^{l})} + \norm{\mathbf{B}_{h}\nabla_{\Gamma}u_{h}}_{L^{2}(K_{h})}\right) + h_{K_{h}}\norm{R}_{L^{2}(w_{e_{h}})} \\
&+ h_{K_{h}}^{1/2}\norm{r - \bar{r}}_{L^{2}(e_{h})} \Bigg). 
\end{align*}
For the jump term in our estimator, we note that since $[u] = 0$ we have
\[\norm{\sqrt{\beta_{e_{h}}}[u_{h}]}_{L^{2}(\partial K_{h})} = \norm{\sqrt{\beta_{e_{h}^{l}}}[u_{h}^{l}]}_{L^{2}(\partial K_{h}^{l})} = \norm{\sqrt{\beta_{e_{h}^{l}}}[u - u_{h}^{l}]}_{L^{2}(\partial K_{h}^{l})} \leq C \norm{u - u_{h}^{l}}_{DG(K_{h}^{l})}.  \]

\end{proof}

\section{Numerical Tests}
In this section we present some numerical tests which verify the reliability and efficiency of the a posteriori estimator given in Theorem \ref{thm:Reliability}. In addition, we look at the benefits of using adaptive refinement for PDEs posed on surfaces and present our own adaptive strategy based on the geometric residual of the estimator. 

\subsection{Implementation Aspects}
All tests are performed using DUNE-FEM, a discretization module based on the Distributed and Unified Numerics Environment (DUNE), (further information about DUNE can be found in \cite{dunegridpaperI:08}, \cite{dunegridpaperII:08} and \cite{dune-web-page}). In all our numerical tests we choose the polynomial order on each element $K_{h} \in \mathcal{T}_{h}$ to be $1$, the penalty parameters to be equal to $10$ and the constant $C$ appearing in the estimator given in Theorem \ref{thm:Reliability} to be equal to $1$. The initial mesh generation for each test case is performed using the 3D surface mesh generation module of the Computational Geometry Algorithms Library (CGAL) (see \cite{rineau20093d}). 

It is worth mentioning that for both test problems discussed below, the lifted point $\xi(x)$ cannot be computed exactly and thus has to be approximated. Details of the algorithm used to do so and further implementational aspects regarding the numerical scheme and the estimator can be found in \cite{demlow2008adaptive} and \cite{dedner2012analysis}.

\subsection{Test Problem on Dziuk Surface}
The first test problem will consider (\ref{eq:EllipticGamma}) on the Dziuk surface, given by $\Gamma = \{ x \in \mathbb{R}^{3}\ : \ (x_{1}-x_{3}^{2})^{2}+x_{2}^{2}+x_{3}^{2}= 1. \}$. As a test solution, we took the function
\[ u(x,y,z) = e^{\frac{1}{1.85 - (x-0.2)^{2}}}\sin y \]
which has sharp gradient changes, as shown in Figure \ref{fig:UniformDziuk1}(a). In Figure \ref{fig:UniformDziuk2}(a) we plot each of the contributions of our error estimator against the number of degrees of freedom when performing global refinement for the Dziuk surface. Note that we plot the \emph{standard residual} with its geometric scaling term i.e.$\left(\sum_{K_{h} \in \mathcal{T}_{h}}\norm{\mathbf{A}_{h}}_{l^{2},L^{\infty}(w_{K_{h}})} \eta_{K_{h}}^2\right)^{1/2}$. Notice how both the geometric residual $\left(\sum_{K_{h} \in \mathcal{T}_{h}}\mathcal{G}_{K_{h}}^2\right)^{1/2}$ and the $DG$ geometric residual $\left(\sum_{K_{h} \in \mathcal{T}_{h}}\mathcal{G}_{DG_{K_{h}}}^2\right)^{1/2}$ converge with higher order as suggested by Lemma \ref{Gamma2GammahSmall}. Figure \ref{fig:UniformDziuk2}(b) confirms that our estimator is efficient, with an efficiency index of about $5.6$.    
\begin{figure}[htp]
\centering
\subfloat[]{\includegraphics[width=0.4 \textwidth]{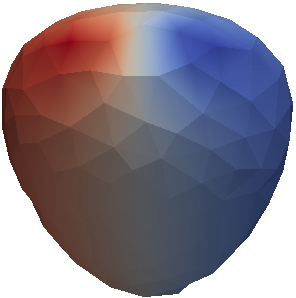}}
\hspace{1cm}
\subfloat[]{\includegraphics[width=0.4 \textwidth]{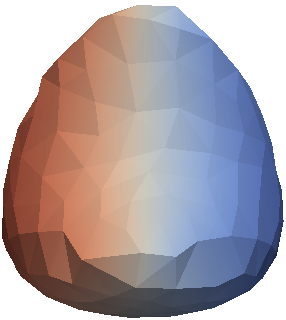}}\\
\caption{Front and rear view of the initial mesh for the Dziuk surface.}
\label{fig:UniformDziuk1}
\subfloat[]{\includegraphics[width=0.50 \textwidth]{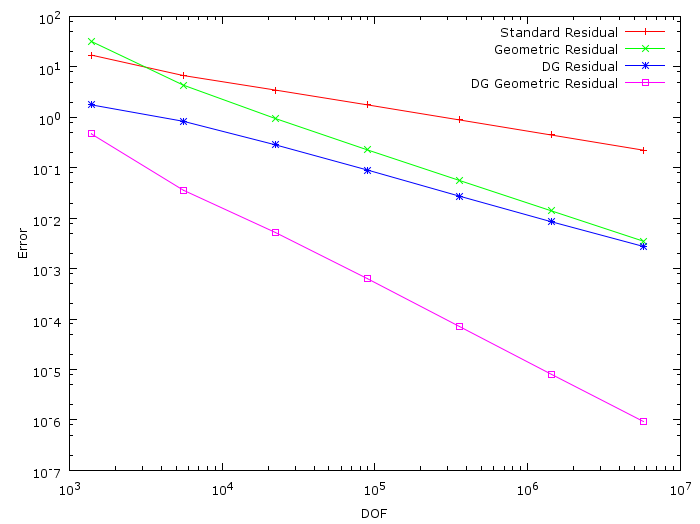}}
\subfloat[]{\includegraphics[width=0.50 \textwidth]{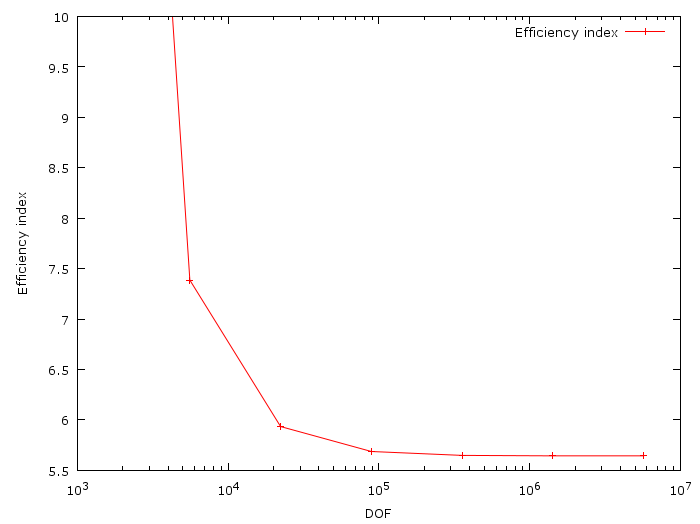}}
\caption{Residual components (left) and efficiency index (right) for the Dziuk surface.}
\label{fig:UniformDziuk2}
\end{figure}

Despite the geometric residual being asymptotically of higher order, it is often the case that initial grids poorly resolve areas of high curvature. This is in fact the case with our initial mesh of the Dziuk surface as can be seen in Figure \ref{fig:UniformDziuk1}(b). Hence, in practice, the geometric residual can be very large for coarser meshes and even remain dominant after multiple global refinements. What we now aim to show is that adaptive refinement strategies based on our estimator are not only useful for problems with sharp changes in the solution, but are also a way of rapidly decreasing the geometric residual for meshes with poorly resolved high curvature areas compared to global refinement. 

Figure \ref{fig:AdaptiveDziuk}(a) shows the plots of the estimator and the true error when performing global and adaptive refinement against the number of degrees of freedom for the Dziuk surface. The adaptive refinement strategy used here is the so-called fixed fraction strategy, detailed for example in Section 3.2 in \cite{rannacher1999posteriori}, with rate $\theta = 0.3$. Notice how the estimator and the true error decrease at a faster rate for coarser meshes when using adaptive refinement, which is due to it rapidly reducing the initially dominant geometric residual. In addition, our estimator appears to attain a given error with approximately a third of the number of degrees of freedom required by global refinement. Figure \ref{fig:AdaptiveDziuk}(b) shows an adaptively refined mesh for the Dziuk surface colour coded by element size. Notice how our estimator captures both the region with exponential peaks (right) and the regions with high curvature (left).     

\begin{figure}[htp]
\centering
\subfloat[]{\includegraphics[width=0.52 \textwidth]{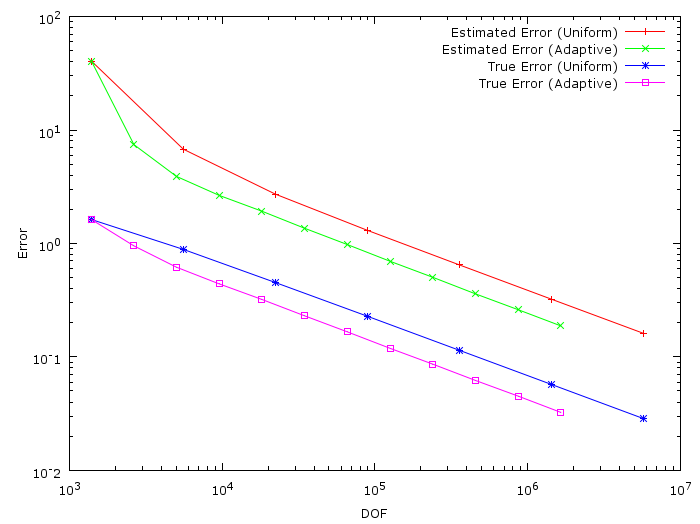}}
\hspace{1cm}
\subfloat[]{\includegraphics[width=0.4 \textwidth]{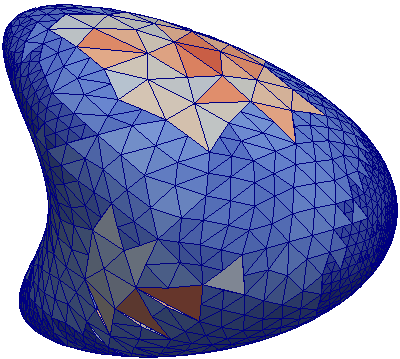}}
\caption{Estimated/true errors for uniform and adaptive refinement (left) and an adaptively refined \\ mesh (right) for the Dziuk surface colour coded by element size.}
\label{fig:AdaptiveDziuk}
\end{figure}           

\subsection{Test Problem on Enzensberger-Stern Surface}
Our second test problem, taken from \cite{dedner2012analysis}, considers (\ref{eq:EllipticGamma}) on the Enzensberger-Stern surface given by  $\Gamma = \{ x \in \mathbb{R}^{3}\ : \ 400(x^2y^2 + y^2z^2 + x^2z^2) - (1-x^2-y^2-z^2)^3 - c = 0 \}$ where $c = 40$ and whose exact solution is chosen to be given by $u(x)=x_{1}x_{2}$. This is a more extreme example of a surface with high curvature areas whose initial mesh poorly resolves them, as shown in Figure \ref{fig:UniformES}(a). In fact, it is worth noting that as $c \rightarrow 0$ the width $\delta_{U}$ of the open subset $U$ required for the one-to-one property of (\ref{eq:uniquePoint}) to hold locally tends to zero. In Figure \ref{fig:UniformES}(b) we plot each of the contributions of our error estimator against the number of degrees of freedom when performing global refinement for the Enzensberger-Stern surface. Notice how the geometric residual term remains the dominant source of error all the way through our computations despite converging with higher order.   
 
\begin{figure}[htp]
\centering
\subfloat[]{\includegraphics[width=0.4 \textwidth]{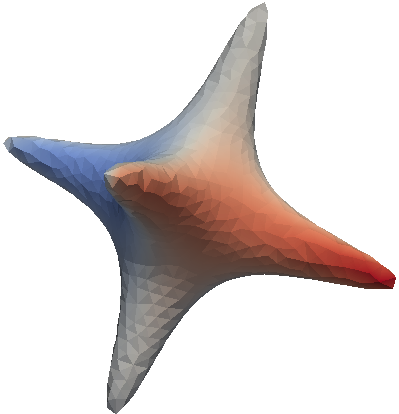}}
\subfloat[]{\includegraphics[width=0.55 \textwidth]{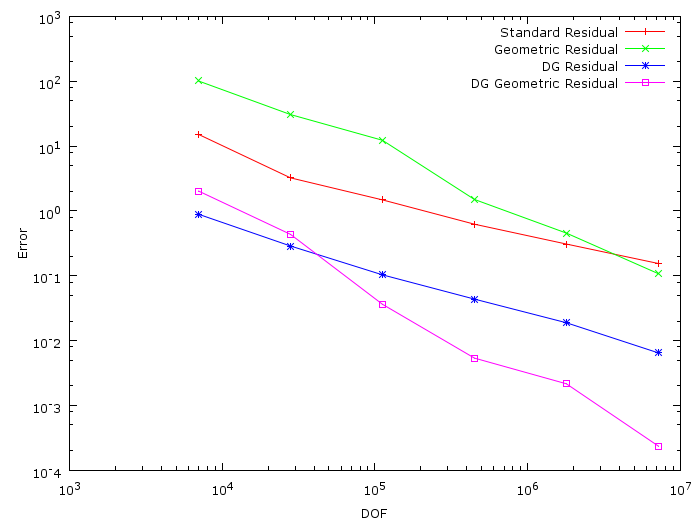}}
\caption{Initial mesh (left) and residual components (right) for the Enzensberger-Stern surface.}
\label{fig:UniformES}
\end{figure}

Figure \ref{fig:AdaptiveES}(a) shows the plots of the estimator and the true error when performing global and adaptive refinement against the number of degrees of freedom for the Enzensberger-Stern surface. The estimator decreases at a much faster rate for coarser meshes when using adaptive refinement by rapidly reducing the geometric residual. 

Figure \ref{fig:AdaptiveES}(b) shows the efficiency of the estimator when performing respectively uniform and adaptive refinement, the latter converging significantly faster to an efficiency index of $5.9$. Figure \ref{fig:AdaptiveES}(c) shows an adaptively refined mesh for the Enzensberger-Stern surface colour coded by element size. Again, our estimator manages to capture the regions of high curvature which were the cause of the dominant geometric residual occuring for global refinement.

We also consider an adaptive refinement strategy based on the geometric residual, as numerics have suggested that it is the dominant contribution for grids that poorly resolve the underlying surface. This strategy only computes the $DG$ approximation $u_{h}$ if the geometric residual statisfies
\[  \frac{\left(\sum_{K_{h} \in \mathcal{T}_{h}}\mathcal{G}_{K_{h}}^2\right)^{1/2}}{\left( \sum_{K_{h} \in \mathcal{T}_{h}}\mathcal{R}_{K_{h}}^{2} + \mathcal{R}_{DG_{K_{h}}}^{2} + \mathcal{G}_{K_{h}}^{2} + \mathcal{G}_{DG_{K_{h}}}^{2} \right)^{1/2} } \leq tol_{\mbox{geometric}} \]
where  $tol_{\mbox{geometric}} \in (0,1)$ is some user-defined tolerance which prescribes how small the geometric residual should be relative to the full estimator. Otherwise, we recompute the estimator and adaptively refine the grid until the criteria is satisfied.  In Figures \ref{fig:AdaptiveES}(a) and \ref{fig:AdaptiveES}(b) we also show respectively the plots of the estimator/true error and the efficiency index when performing our geometric adaptive refinement strategy. Highlighted are the iterations at which the $DG$ approximation is recomputed; the true error is only plotted for those iterations. Our estimator reaches a similar error as the standard adaptive strategy as we increase the number of degrees of freedom but requires far less recomputations of the $DG$ approximation ($11$ for the standard adaptive strategy compared to $5$ for the geometric adaptive strategy), hence significantly more computationally efficient. It is also worth mentioning that although we do not have a rigorous proof that the stopping criteria for our geometric adaptive refinement strategy would be satisfied, it appears that this is in fact the case for all of our test problems, with the number of iterations required to satisfy the stopping criteria decreasing as expected.      
Note also that after a number of refinement steps the curves for both
refinement strategies seem to collapse but that we are in fact reaching the
same error with slightly fewer elements in addition to requiring fewer
computations of $u_h$.

\begin{figure}[htp]
\centering
\subfloat[]{\includegraphics[width=0.45 \textwidth]{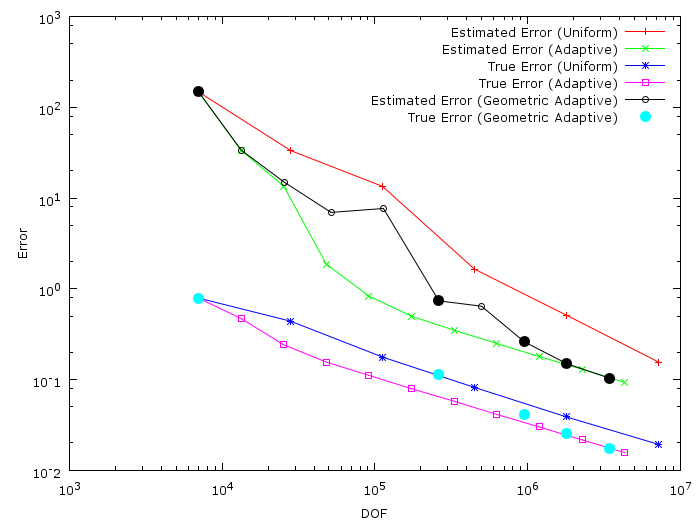}}
\hspace{1cm}
\subfloat[]{\includegraphics[width=0.45 \textwidth]{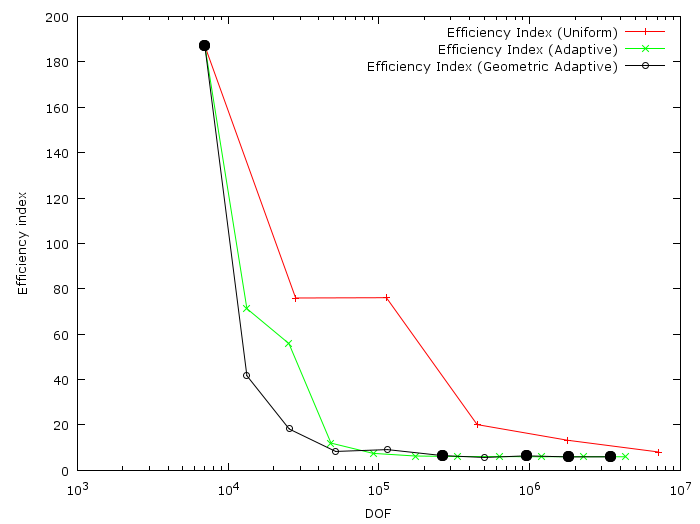}} \\
\subfloat[]{\includegraphics[width=0.95 \textwidth]{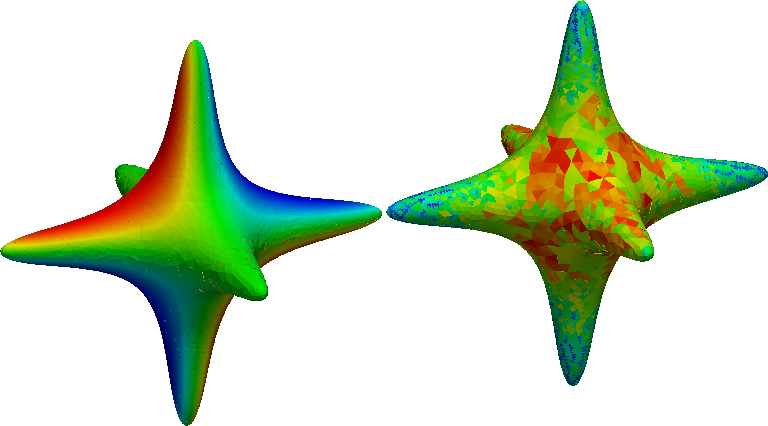}} \\
\caption{Estimated/true errors (top right) and efficiency indices (top left) for uniform and adaptive \\ refinement.
Results for both standard and geometric adaptation strategies are shown.
The solution \\ and a color coding of the adaptive mesh are shown in the bottom row.
}
\label{fig:AdaptiveES}
\end{figure}

\section{Conclusions}
In this paper, we derived a dual weighted residual-based a posteriori error estimate for a surface DG discretisation of a model second-order elliptic problem posed on a smooth surface in $\mathbb{R}^{3}$. We proved both reliability and efficiency of the error estimator in the energy norm and showed that the error may be split into a ``residual part'', made up of the standard resdidual term along with the jump of the DG approximation, and a higher order ``geometric part'' which arises from the lack of Galerkin orthogonality. These were then verified numerically for a number of test problems and, in the process, we showed the benefits of using adaptive refinement for problems on surfaces with poorly resolved regions of high curvature. We then described and tested an adaptive refinement strategy which was based on the ``geometric part'' of the residual and showed that we may obtain similar errors to the standard adaptive refinement strategy for a fraction of the computational cost. We have recently derived higher-order a priori error estimates for a large class of surface DG methods in \cite{antonietti2013analysis} and will naturally be looking at extending our estimator to encorporate both anisotropic DG space order and surface approximation order in the future.      

\section*{Acknowledgements}
We would like to thank Dr. Alan Demlow and Dr. Bj\"{o}rn Stinner for their useful comments and suggestions. This research has been supported by the British Engineering and Physical Sciences Research Council (EPSRC), Grant EP/H023364/1.

\bibliographystyle{IMANUM-BIB}
\bibliography{IMANUM-refs}

\begin{thebibliography}{}

\bibitem[Ainsworth \& Oden(2011)Ainsworth \& Oden]{ainsworth2011posteriori}
{\sc Ainsworth, M. \& Oden, J.~T.} (2011)
\newblock {\em A posteriori error estimation in finite element analysis\/},
  vol.~37.
\newblock John Wiley \& Sons.

\bibitem[Antonietti {\em et~al.}(2013)Antonietti, Dedner, Madhavan, Stangalino,
  Stinner, \& Verani]{antonietti2013analysis}
{\sc Antonietti, P., Dedner, A., Madhavan, P., Stangalino, S., Stinner, B. \&
  Verani, M.} (2013)
\newblock Higher order discontinuous galerkin methods on surfaces.
\newblock {\em In preparation\/}.

\bibitem[Arnold {\em et~al.}(2002)Arnold, Brezzi, Cockburn, \&
  Marini]{arnold2002unified}
{\sc Arnold, D., Brezzi, F., Cockburn, B. \& Marini, L.} (2002)
\newblock Unified analysis of discontinuous galerkin methods for elliptic
  problems.
\newblock {\em SIAM journal on numerical analysis\/}, 1749--1779.

\bibitem[Aubin(1982)Aubin]{aubin1982nonlinear}
{\sc Aubin, T.} (1982)
\newblock {\em Nonlinear analysis on manifolds, Monge-Ampere equations\/},
  vol. 252.
\newblock Springer.

\bibitem[Bastian {\em et~al.}(2008a)Bastian, Blatt, Dedner, Engwer, Kl\"ofkorn,
  Ohlberger, \& Sander]{dunegridpaperI:08}
{\sc Bastian, P., Blatt, M., Dedner, A., Engwer, C., Kl\"ofkorn, R., Ohlberger,
  M. \& Sander, O.} (2008a)
\newblock {A Generic Grid Interface for Parallel and Adaptive Scientific
  Computing. Part {I}: Abstract Framework}.
\newblock {\em Computing\/}, {\bf 82}, 103--119.

\bibitem[Bastian {\em et~al.}(2008b)Bastian, Blatt, Dedner, Engwer, Kl\"ofkorn,
  Kornhuber, Ohlberger, \& Sander]{dunegridpaperII:08}
{\sc Bastian, P., Blatt, M., Dedner, A., Engwer, C., Kl\"ofkorn, R., Kornhuber,
  R., Ohlberger, M. \& Sander, O.} (2008b)
\newblock {A Generic Grid Interface for Parallel and Adaptive Scientific
  Computing. Part {II}: Implementation and Tests in {DUNE}}.
\newblock {\em Computing\/}, {\bf 82}, 121--138.

\bibitem[Bastian {\em et~al.}(2012)Bastian, Blatt, Dedner, Engwer, Fahlke,
  Gr{\"a}ser, Kl{\"o}fkorn, Nolte, Ohlberger, \& Sander]{dune-web-page}
{\sc Bastian, P., Blatt, M., Dedner, A., Engwer, C., Fahlke, J., Gr{\"a}ser,
  C., Kl{\"o}fkorn, R., Nolte, M., Ohlberger, M. \& Sander, O.}
\newblock  (2012).
\newblock http://www.dune-project.org.

\bibitem[Cockburn {\em et~al.}(2000)Cockburn, Karniadakis, \&
  Shu]{cockburn2000development}
{\sc Cockburn, B., Karniadakis, G. \& Shu, C.} (2000)
\newblock The development of discontinuous galerkin methods.
\newblock {\em UMSI research report/University of Minnesota (Minneapolis, Mn).
  Supercomputer institute\/}, {\bf 99}, 220.

\bibitem[Cockburn(2003)Cockburn]{cockburn2003discontinuous}
{\sc Cockburn, B.} (2003)
\newblock Discontinuous galerkin methods.
\newblock {\em ZAMM-Journal of Applied Mathematics and Mechanics/Zeitschrift
  f{\"u}r Angewandte Mathematik und Mechanik\/}, {\bf 83}, 731--754.

\bibitem[Deckelnick {\em et~al.}(2001)Deckelnick, Elliott, \&
  Styles]{DecEllSty01}
{\sc Deckelnick, K., Elliott, C. \& Styles, V.} (2001)
\newblock Numerical diffusion-induced grain boundary motion.
\newblock {\em Interfaces Free Bound.}, {\bf 3}, 393--414.

\bibitem[Deckelnick {\em et~al.}(2005)Deckelnick, Dziuk, \&
  Elliott]{deckelnick2005computation}
{\sc Deckelnick, K., Dziuk, G. \& Elliott, C.} (2005)
\newblock Computation of geometric partial differential equations and mean
  curvature flow.
\newblock {\em Acta Numerica\/}, {\bf 14}, 139--232.

\bibitem[Dedner {\em et~al.}(2010)Dedner, Kl\"ofkorn, Nolte, \&
  Ohlberger]{dunefempaper:10}
{\sc Dedner, A., Kl\"ofkorn, R., Nolte, M. \& Ohlberger, M.} (2010)
\newblock {A Generic Interface for Parallel and Adaptive Scientific Computing:
  Abstraction Principles and the DUNE-FEM Module}.
\newblock {\em Computing\/}, {\bf 90}, 165--196.

\bibitem[Dedner {\em et~al.}(2013)Dedner, Madhavan, \&
  Stinner]{dedner2012analysis}
{\sc Dedner, A., Madhavan, P. \& Stinner, B.} (2013)
\newblock Analysis of the discontinuous galerkin method for elliptic problems
  on surfaces.
\newblock {\em IMA Journal of Numerical Analysis\/}.

\bibitem[Demlow \& Dziuk(2008)Demlow \& Dziuk]{demlow2008adaptive}
{\sc Demlow, A. \& Dziuk, G.} (2008)
\newblock An adaptive finite element method for the laplace-beltrami operator
  on implicitly defined surfaces.
\newblock {\em SIAM Journal on Numerical Analysis\/}, {\bf 45}, 421--442.

\bibitem[Dziuk(1988)Dziuk]{dziuk1988finite}
{\sc Dziuk, G.} (1988)
\newblock Finite elements for the beltrami operator on arbitrary surfaces.
\newblock {\em Partial differential equations and calculus of variations\/},
  142--155.

\bibitem[Dziuk \& Elliott(2007a)Dziuk \& Elliott]{dziuk2007finite}
{\sc Dziuk, G. \& Elliott, C.} (2007a)
\newblock Finite elements on evolving surfaces.
\newblock {\em IMA journal of numerical analysis\/}, {\bf 27}, 262.

\bibitem[Dziuk \& Elliott(2007b)Dziuk \& Elliott]{dziuk2007surface}
{\sc Dziuk, G. \& Elliott, C.} (2007b)
\newblock Surface finite elements for parabolic equations.
\newblock {\em J. Comput. Math\/}, {\bf 25}, 385--407.

\bibitem[Dziuk \& Elliott(2013)Dziuk \& Elliott]{dziuk2013finite}
{\sc Dziuk, G. \& Elliott, C.~M.} (2013)
\newblock Finite element methods for surface pdes.
\newblock {\em Acta Numerica\/}, {\bf 22}, 289--396.

\bibitem[Elliott \& Stinner(2010)Elliott \& Stinner]{EllSti10}
{\sc Elliott, C. \& Stinner, B.} (2010)
\newblock Modeling and computation of two phase geometric biomembranes using
  surface finite elements.
\newblock {\em J. Comp. Phys.}, {\bf 229}, 6585--6612.

\bibitem[Giesselmann \& M{\"u}ller(2013)Giesselmann \& M{\"u}ller]{GieMue_prep}
{\sc Giesselmann, J. \& M{\"u}ller, T.} (2013)
\newblock Geometric error of finite volume schemes for conservation laws on
  evolving surfaces.
\newblock {\em arXiv preprint arXiv:1301.1287\/}.

\bibitem[Houston {\em et~al.}(2007)Houston, Schotzau, Wihler, \&
  Schwab]{houston2007energy}
{\sc Houston, P., Schotzau, D., Wihler, T. \& Schwab, C.} (2007)
\newblock Energy norm a posteriori error estimation of hp-adaptive
  discontinuous galerkin methods for elliptic problems.
\newblock {\em Mathematical Models and Methods in Applied Sciences\/}, {\bf
  17}, 33--62.

\bibitem[James \& Lowengrub(2004)James \& Lowengrub]{JamLow04}
{\sc James, A. \& Lowengrub, J.} (2004)
\newblock A surfactant-conserving volume-of-fluid method for interfacial flows
  with insoluble surfactant.
\newblock {\em J. Comp. Phys.}, {\bf 201}, 685--722.

\bibitem[Ju {\em et~al.}(2009)Ju, Tian, \& Wang]{ju2009posteriori}
{\sc Ju, L., Tian, L. \& Wang, D.} (2009)
\newblock A posteriori error estimates for finite volume approximations of
  elliptic equations on general surfaces.
\newblock {\em Computer Methods in Applied Mechanics and Engineering\/}, {\bf
  198}, 716--726.

\bibitem[Ju \& Du(2009)Ju \& Du]{ju2009finite}
{\sc Ju, L. \& Du, Q.} (2009)
\newblock A finite volume method on general surfaces and its error estimates.
\newblock {\em Journal of Mathematical Analysis and Applications\/}, {\bf 352},
  645--668.

\bibitem[Karakashian \& Pascal(2003)Karakashian \&
  Pascal]{karakashian2003posteriori}
{\sc Karakashian, O.~A. \& Pascal, F.} (2003)
\newblock A posteriori error estimates for a discontinuous galerkin
  approximation of second-order elliptic problems.
\newblock {\em SIAM Journal on Numerical Analysis\/}, {\bf 41}, 2374--2399.

\bibitem[Larsson \& Larson(2013)Larsson \& Larson]{larsson2013continuous}
{\sc Larsson, K. \& Larson, M.~G.} (2013)
\newblock A continuous/discontinuous galerkin method and a priori error
  estimates for the biharmonic problem on surfaces.
\newblock {\em arXiv preprint arXiv:1305.2740\/}.

\bibitem[Lenz {\em et~al.}(2011)Lenz, Nemadjieu, \& Rumpf]{lenz2011convergent}
{\sc Lenz, M., Nemadjieu, S.~F. \& Rumpf, M.} (2011)
\newblock A convergent finite volume scheme for diffusion on evolving surfaces.
\newblock {\em SIAM Journal on Numerical Analysis\/}, {\bf 49}, 15--37.

\bibitem[Mekchay {\em et~al.}(2011)Mekchay, Morin, \&
  Nochetto]{mekchay2011afem}
{\sc Mekchay, K., Morin, P. \& Nochetto, R.} (2011)
\newblock Afem for the laplace-beltrami operator on graphs: design and
  conditional contraction property.
\newblock {\em Mathematics of Computation\/}, {\bf 80}, 625--648.

\bibitem[Neilson {\em et~al.}(2011)Neilson, Mackenzie, Webb, \&
  Insall]{neilson2011modelling}
{\sc Neilson, M., Mackenzie, J., Webb, S. \& Insall, R.} (2011)
\newblock Modelling cell movement and chemotaxis pseudopod based feedback.
\newblock {\em SIAM Journal on Scientific Computing\/}, {\bf 33}.

\bibitem[Rannacher \& Suttmeier(1999)Rannacher \&
  Suttmeier]{rannacher1999posteriori}
{\sc Rannacher, R. \& Suttmeier, F.-T.} (1999)
\newblock A posteriori error estimation and mesh adaptation for finite element
  models in elasto-plasticity.
\newblock {\em Computer methods in applied mechanics and engineering\/}, {\bf
  176}, 333--361.

\bibitem[Rineau \& Yvinec(2009)Rineau \& Yvinec]{rineau20093d}
{\sc Rineau, L. \& Yvinec, M.} (2009)
\newblock 3d surface mesh generation.
\newblock {\em CGAL Editorial Board, editor, CGAL User and Reference Manual\/},
  {\bf 3}, 53.

\bibitem[Sch{\"o}tzau {\em et~al.}(2003)Sch{\"o}tzau, Schwab, \&
  Toselli]{schotzau2003mixed}
{\sc Sch{\"o}tzau, D., Schwab, C. \& Toselli, A.} (2003)
\newblock Mixed hp-dgfem for incompressible flows.
\newblock {\em SIAM Journal on Numerical Analysis\/}, 2171--2194.

\bibitem[Sch{\"o}tzau \& Zhu(2009)Sch{\"o}tzau \& Zhu]{schotzau2009robust}
{\sc Sch{\"o}tzau, D. \& Zhu, L.} (2009)
\newblock A robust a-posteriori error estimator for discontinuous galerkin
  methods for convection--diffusion equations.
\newblock {\em Applied numerical mathematics\/}, {\bf 59}, 2236--2255.

\bibitem[Verf{\"u}rth(1989)Verf{\"u}rth]{verfurth1989posteriori}
{\sc Verf{\"u}rth, R.} (1989)
\newblock A posteriori error estimators for the stokes equations.
\newblock {\em Numerische Mathematik\/}, {\bf 55}, 309--325.

\bibitem[Wloka(1987)Wloka]{wlokapartial}
{\sc Wloka, J.} (1987)
\newblock Partial differential equations.
\newblock {\em Cambridge University\/}.

\end{thebibliography}

\end{document}